\documentclass[11pt,a4paper, envcountsame]{amsart}
\usepackage[usenames,dvipsnames]{color}

\usepackage[usenames,dvipsnames]{color}

 \usepackage[colorlinks,citecolor=blue,urlcolor=blue, linkcolor=blue]{hyperref}

 \usepackage{amsmath, amssymb, xspace}
 \usepackage{graphicx}

 \usepackage[all]{xy}
 
 \usepackage{amsmath,amssymb,amscd,amsthm,amsfonts,amstext,amsbsy,mathrsfs,hyperref,upgreek,mathtools,stmaryrd,enumitem,bbm}

\usepackage[dvipsnames]{xcolor} 
\definecolor{lightblue}{RGB}{0,190,255}
\usepackage[textsize=footnotesize,color=lightblue, bordercolor=white]{todonotes}

\usepackage{tikz}
\usetikzlibrary{arrows,snakes,positioning,backgrounds,shadows}

\newtheorem{theorem}{Theorem}[section]

\newtheorem{thm}[theorem]{Theorem}

\newtheorem{fact}[theorem]{Fact}

\newtheorem{prop}[theorem]{Proposition}

\newtheorem{lemma}[theorem]{Lemma}

\newtheorem{cor}[theorem]{Corollary}

\theoremstyle{definition}
\newtheorem{definition}[theorem]{Definition}

\newtheorem{claim}[theorem]{Claim}
\newtheorem{remark}[theorem]{Remark}
\newtheorem{example}[theorem]{Example}

\newcommand{\NN}{{\mathbb{N}}}
\newcommand{\RR}{{\mathbb{R}}}

\newcommand{\sub}{\subseteq}
\newcommand{\sN}[1]{_{#1\in \omega}}

\newcommand{\bi}{\begin{itemize}}
\newcommand{\ei}{\end{itemize}}
\newcommand{\bc}{\begin{center}}
\newcommand{\ec}{\end{center}}

\newcommand{\ES}{\emptyset}

\newcommand{\ex}{\exists}
\newcommand{\fa}{\forall}

\newcommand{\la}{\langle}
\newcommand{\ra}{\rangle}

\newcommand{\strcantor}{2^{ < \omega}}

\newcommand{\n}{\noindent}

\newcommand{\sss}{\sigma}

\newcommand{\omc}{L_{\omega_1, \omega}}

\newcommand{\lland}{\, \land \, }

\newcommand \seq[1]{{\left\langle{#1}\right\rangle}}

\newcommand\+[1]{\mathcal{#1}}

\newcommand{\wt}{\widetilde}
\newcommand{\ol}{\overline}

\newcommand{\lra}{\leftrightarrow}
\newcommand{\LR}{\Leftrightarrow}
\newcommand{\RA}{\Rightarrow}
\newcommand{\LA}{\Leftarrow}

\newcommand{\rapf}{\n $\RA:$\ }
\newcommand{\lapf}{\n $\LA:$\ }

\newcommand{\sssl}{\ensuremath{|\sigma|}}

\def \sq {\sqsubseteq}

\DeclareMathOperator \Th{Th}

  \DeclareMathOperator{\Aut}{Aut}
   \DeclareMathOperator{\GI}{GI}
\renewcommand{\S}{S_\infty}
\renewcommand{\hat}{\widehat}

\newtheorem{axiom}{Axiom}

\newcommand{\F}{\mathcal{F}} 
\newcommand{\G}{\mathcal{G}} 
\newcommand{\ran}{\mathrm{ran}}

\newcommand{\U}{\mathcal{U}}

\newenvironment{enumerate-(a)}{\begin{enumerate}[label={\upshape (\alph*)}, leftmargin=2pc]}{\end{enumerate}}
\newenvironment{enumerate-(a)-r}{\begin{enumerate}[label={\upshape (\alph*)}, leftmargin=2pc,resume]}{\end{enumerate}}
\newenvironment{enumerate-(A)}{\begin{enumerate}[label={\upshape (\Alph*)}, leftmargin=2pc]}{\end{enumerate}}
\newenvironment{enumerate-(A)-r}{\begin{enumerate}[label={\upshape (\Alph*)}, leftmargin=2pc,resume]}{\end{enumerate}}
\newenvironment{enumerate-(i)}{\begin{enumerate}[label={\upshape (\roman*)}, leftmargin=2pc]}{\end{enumerate}}
\newenvironment{enumerate-(i)-r}{\begin{enumerate}[label={\upshape (\roman*)}, leftmargin=2pc,resume]}{\end{enumerate}}
\newenvironment{enumerate-(I)}{\begin{enumerate}[label={\upshape (\Roman*)}, leftmargin=2pc]}{\end{enumerate}}
\newenvironment{enumerate-(I)-r}{\begin{enumerate}[label={\upshape (\Roman*)}, leftmargin=2pc,resume]}{\end{enumerate}}
\newenvironment{enumerate-(1)}{\begin{enumerate}[label={\upshape (\arabic*)}, leftmargin=2pc]}{\end{enumerate}}
\newenvironment{enumerate-(1)-r}{\begin{enumerate}[label={\upshape (\arabic*)}, leftmargin=2pc,resume]}{\end{enumerate}}

\subjclass{20A15,  03E15}

\begin{document}

%
 


 \title[Coarse groups and their applications]{Coarse groups, and the isomorphism problem for oligomorphic groups}
\author{Andr\'e Nies}
\author{Philipp Schlicht}
\author{Katrin Tent}

\thanks{The first and second author were partially supported by the Marsden fund of New Zealand, 13-UOA-184 and 19-UOA-346. 
This project has received funding from the European Union's Horizon 2020 research and innovation programme under the Marie Sk\l odowska-Curie grant agreement No 794020 (IMIC) of the second author. 
The second author was partially supported by FWF grant number I4039 during the revision of this paper. 
The third author was supported by the DFG, under CRC 878  and under Germany's Excellence Strategy -EXC 2044-, Mathematics M\"unster: Dynamics-Geometry-Structure. The authors gratefully acknowledge the input of  the referee who has made very insightful comments that improved the paper.
}

\noindent 
\address{A.\  Nies, School of Computer Science,  The University of Auckland, Private Bag 92019, Auckland 1142. \newline  \texttt{andre@cs.auckland.ac.nz}}

\address{P.\ Schlicht,  School of Mathematics, University of Bristol, Fry Building, Woodland Road, Bristol, BS8 1UG, and 
Mathematisches Institut, Universit\"at Bonn, Endenicher Allee 60, 53155 Bonn. 
\newline \texttt{schlicht@math.uni-bonn.de}}

\address{K.\ Tent, Mathematisches Institut, Einsteinstrasse 62, Universit\"at M\"unster, 48149 M\"unster. \newline \texttt{tent@wwu.de}}
\maketitle

\begin{abstract} Let $S_\infty$ denote  the topological group of permutations of the natural numbers.    A closed subgroup $G$ of $S_\infty$ is called \emph{oligomorphic}   if  for each $n$,  its natural action on  $n$-tuples of natural numbers has only finitely many orbits. 
 We study the complexity of the topological  isomorphism relation on the oligomorphic subgroups of $S_\infty$ in the setting of Borel reducibility between equivalence relations on Polish spaces.

 Given a closed subgroup $G$ of $S_\infty$, the  \emph{coarse group} $\mathcal M(G)$      is the    structure with domain the  cosets  of   open subgroups of $G$, and a     ternary relation $AB \sub C$. 
 This structure derived from $G$ was introduced   in \cite[Section 3.3]{Kechris.Nies.etal:18}.
    If $G$ has   only  countably many open subgroups, then  $\mathcal M(G)$ is a  countable structure.   Coarse groups form   our main   tool in studying such closed subgroups of $S_\infty$. We axiomatise them abstractly as structures with a ternary relation. For the oligomorphic groups,  and also the     profinite groups,     we set up a  Stone-type duality between the groups and the corresponding coarse groups. In particular we can recover an isomorphic copy of~$G$ from its coarse group in a Borel fashion.
 
 We use this duality to  show that the isomorphism relation for oligomorphic  subgroups of $S_\infty$   is Borel reducible to a  Borel equivalence relation with all classes countable.  We show that  the same upper bound applies to the larger class of  closed subgroups of $S_\infty$ that are  topologically isomorphic to oligomorphic  groups.
  \end{abstract}

%
%

\section{Introduction}
 \n  Let $\S$ denote the Polish  group of permutations of the natural numbers with the usual  topology of pointwise convergence.     The closed subgroups of~$\S$ (also called non-Archimedean groups) form a standard Borel space. All the   classes  of groups we consider will   be  Borel sets in this space that are invariant under conjugation by elements of $\S$.     (Details will be provided in Section~\ref{s:prelim}.) 
  
    Kechris and two of the authors~\cite{Kechris.Nies.etal:18}  determined the complexity of the topological isomorphism relation on certain classes  closed subgroups of $\S$. They  used the setting of Borel reducibility between equivalence relations $E$ and $F$ on Borel  spaces $X$ and $Y$, respectively: $E$ is Borel reducible to $F$, written   $E\le_B F$,  if there is  $f\colon X \to Y$ such that the preimage of any Borel set in $Y$ is Borel  in $X$, and $x_0 E x_1 \LR f(x_0) F f(x_1)$ for each $x_0, x_1 \in X$. See e.g.\ \cite{Gao:09} for background on Borel reducibility.

    In this paper, all topological groups will be separable and all  isomorphisms between them will be topological (that is, both the isomorphism and its inverse are continuous). One    result in Kechris et al.~\cite{Kechris.Nies.etal:18}  addresses the compact subgroups of $\S$; note that 
    these  are the    separable profinite groups. Their result 
states that the  isomorphism relation for compact  subgroups  of $\S$  is Borel equivalent to the isomorphism relation between  countable graphs. In particular, it is properly analytic.  

A closed subgroup $G$ of $\S$ is called \emph{oligomorphic} (see~\cite{Cameron:90}) if  for each positive natural number~$n$,  its canonical  action on $\omega^n$, the set of $n$-tuples of natural numbers,  has only finitely many orbits (these will be  called $n$-orbits). 
Note that this is not a group theoretic property; rather, it  depends on the group 
action  and hence on the embedding of the group into $S_\infty$. 
The oligomorphic groups  are precisely the automorphism groups of $\omega$-categorical structures with domain the set of natural numbers.   They are, in a sense, opposite to compact subgroups of~$\S$, which are characterised by the condition that for each $n$, each $n$-orbit is finite. For background on oligomorphic 
groups we refer the reader to~\cite{Cameron:90},  and also to Tsankov~\cite{Tsankov:12}.

We  show that the isomorphism relation between   oligomorphic groups is far below graph isomorphism: it is Borel reducible to a Borel equivalence relation with all classes countable. This property of an equivalence relation on a Polish space  is called ``essentially countable".

 Closed subgroups of $\S$ that are isomorphic to oligomorphic groups will be called  \emph{quasi-oligomorphic}. Near the end of the paper we will show that this class is Borel, and   that  the  same upper bound on the isomorphism relation also applies to  this   class. 

While oligomorphic and compact subgroups of $\S$ are at opposite ends of the spectrum, they have   a common superclass.
 A  Polish group $G$ is called \emph{Roelcke precompact} if for every neighborhood of the identity $U$, there exists a finite set $F\subseteq G$ such that $G = UFU$. In other words, the  equivalence relation $ \sim_U =  \{\la x,y \ra \colon \exists u,v \in U \, uxv= y \} $ has only  finitely many equivalence classes. Roelcke precompactness of   closed subgroups of $\S$ is a Borel property as noted in~\cite{Kechris.Nies.etal:18}. 
It is  well-known that \emph{every  Roelcke precompact   group $G$ has only countably many open subgroups.} (To see this, let $U_n$ denote the pointwise stabiliser of $\{0, \ldots, n\}$ in   $G$. Each open subgroup $U$ of $G$ contains a group  $U_n$, and hence is a finite union of $\sim_{U_n}$ classes.  So there are only countably many possibilities for $U$.)

  Figure~\ref{fig:diagram} summarises the Borel reductions between isomorphism relations obtained in the earlier reference~\cite{Kechris.Nies.etal:18}  and  the present paper. The wavy arrows indicate known Borel reductions; unreferenced arrows are trivial ``identity'' reductions given by the inclusion of Borel classes.

It is well known that there are uncountably many non-isomorphic oligomorphic groups. For instance, Evans and Hewitt \cite[Lemma 3.1 and its proof]{Evans.Hewitt:90} show that each profinite group $K$  is isomorphic to a group of the form $\Sigma/\Phi$, where $\Sigma$ is an oligomorphic group and $\Phi$ is the intersection of all its  open subgroups of finite index. (In  their construction, one can see how $\Sigma $ depends  on the way $K$ is presented as a subgroup of $\S$.  So, one does not obtain  $\Sigma$ from $K$ through   a Borel function that preserves isomorphism; note that this would contradict our   result that isomorphism of oligomorphic groups is essentially countable.) 
Alternatively, there are uncountably many pairwise non-isomorphic 
automorphism groups of Henson digraphs~\cite{Henson:72}: 
if $\Aut(G)$, $\Aut(H)$ are isomorphic automorphism groups of Henson digraphs, then  
they are conjugate by \cite[Example 1 in Section 3, Theorem 2.2 \& Theorem 3.2]{Rubin:94}, 
thus inducing an isomorphism $G\cong H$ or $G \cong H^{-1}=\{(x,y)\mid (y,x)\in G\}$ by ultrahomogeneity of the digraphs. 

 We leave open the question whether there is a lower bound for $\cong_{\text{oligomorphic}}$ that is higher  than   the identity on $\RR$. 
 This question may have a negative answer when we require in addition that the signature of the corresponding canonical structures is finite up to interdefinability (see Subsection 1.3).

\begin{figure} \label{fig:diagram}

\[ \xymatrix{    & { \GI} & \\   
E_\infty\ar@{~>}[ur]^{<_B} & \cong_\text{Roelcke precompact}\ar@{~>}[u]^{\equiv_B}_{\text{\cite[Thm.\  3.1(ii)]{Kechris.Nies.etal:18}}} & \\    
\cong_\text{quasi-oligomorphic}\ar@{~>}[ur]^{\le_B}\ar@{~>}[u]_{\le_B}^{\text{present paper}}    &   & \cong_\text{compact}\ar@{~>}[ul]_{\equiv_B}   \\ 
   \cong_\text{oligomorphic}\ar@{<~>}[u]_{\equiv_B}^{\text{present paper}}   & &  {\GI}\ar@{~>}[u]^{\equiv_B}_{\text{\cite[Thm.\  4.3]{Kechris.Nies.etal:18}}}} \]

%
   
\caption{\small Borel reductions between isomorphism relations. $E_\infty$~denotes a  $\le_B$-complete countable Borel equivalence relation.
$\GI$~denotes isomorphism of countable graphs, which  is $\le_B$-complete for orbit equivalence relations given by continuous actions of $\S$.}
\end{figure}
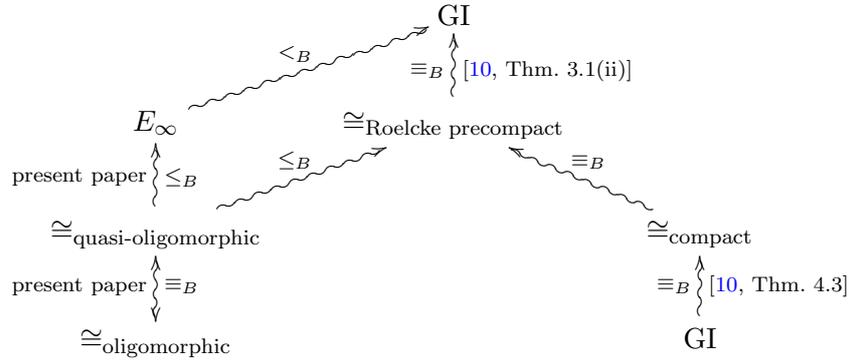

 \subsection{The coarse group $\+ M(G)$ associated with $G$} \label{s:CG}
  Coarse groups were  introduced by Kechris, Nies and Tent \cite[Section 3.3]{Kechris.Nies.etal:18}, in order to provide an alternative proof of their main result, that there are Borel reductions of  the  isomorphism relation for  Roelcke precompact groups, and  also for totally disconnected locally compact (t.d.l.c.)   groups,  to the isomorphism relation between    graphs with domain $\omega$. 
  
  First we recall  a few  preliminaries. A Polish group  is isomorphic to a   closed subgroup of $\S$ if and only if its  neutral element has a neighbourhood basis consisting of open subgroups; see e.g.~\cite[Thm.\ 1.5.1]{Becker.Kechris:96}. Note that each   left coset $aU$ of an open subgroup $U$ is also a right coset of the open subgroup $aUa^{-1}$. We will use the term \emph{open coset} for some coset, left or right, of an open subgroup. $G$ will usually denote a closed subgroup of~$\S$. The open cosets in $G$  form a left,  and also right,  translation invariant base for the subspace  topology on $G$.   We will use letters $A,B,C,D$ to denote open cosets.

 The domain of the  coarse group $\+ M(G)$ associated with $G$ consists of  the open cosets. Instead of the binary group operation,   it  has a  ternary relation $AB \sub C$. If $xy = z$ in $G$ then by continuity, for each $C \ni z$ there are $A \ni x $ and $B \ni y$ such that $AB \sub C$. So this ternary relation  approximates the group operation. 
       Kechris, Nies and Tent \cite[Section 3.3]{Kechris.Nies.etal:18} assigned   to a Roelcke precompact group $G$   in a canonical, Borel   way  an isomorphic  copy   of the structure $\+ M(G)$ with domain the natural numbers, and  showed that  
 for Roelcke precompact closed subgroups $G,H$ of $\S$, one has
\begin{equation} \label{eqn:KNS} G \cong H \LR \+ M(G) \cong \+ M(H). \end{equation}
Since the coarse groups can be assumed  to  have domain $\omega$, by standard coding techniques this implies that isomorphism of Roelcke precompact  groups is Borel reducible to isomorphism of countable graphs.

 In Section~\ref{s:coarse} we will axiomatise the basic properties of an abstract coarse group $M$. We provide  axioms that govern  subgroups, inclusion, and allow us to  define an operation $B= A^\diamond$ approximating the inverse operation in a group. We introduce  the filter group $\+ F(M)$, which consists of the filters that contain a (unique) coset of each subgroup. Our main interest is  in the case that $M$ is countable, in which case we show that $\+ F(M)$ is a Polish totally disconnected  group. If $G$ has  only countably many open subgroups (such as when $G$ is Roelcke precompact), then we can recover $G$ from its coarse group in the sense that $G \cong \+ F (\+ M(G))$.

%

\subsection{Borel duality of  oligomorphic groups with their coarse groups} \label{rem:plan} 
Let $\+ B$ be  the closure  under isomorphism of the  range of the operator $\+ M$ on 
the class of oligomorphic groups. 
Theorem~\ref{thm:main} will show that  $\+ B$  is Borel, and that   on $\+ B$ one can define  a Borel   operator $\+ G$       that is an   ``inverse up to isomorphism"  of   $\+ M$, in the sense that \bi \item  $\+ G( \+ M(G)) \cong G$ for each oligomorphic $G$, and \item  $\+ M( \+ G(M)) \cong M$ for each $M \in \+ B$.    \ei
 Using (\ref{eqn:KNS}) it follows that \bc  $ M \cong N \LR \+ G(M) \cong \+ G(N)$. \ec So,  in a Borel  fashion  we can ``interchange'' oligomorphic groups   with  their coarse groups, which are  countable structures. 
 
Theorem~\ref{thm:main} will  be stated via a  notion introduced in \cite{Friedman.MottoRos:11} as \emph{classwise Borel isomorphism}; also see   \cite[Def. 2.1]{MottoRos:12}. (We have slightly adapted the terminology here to avoid using ``isomorphism" in two different ways in the same sentence.)  Recall that a standard Borel space \cite[Section 12B]{Kechris:95} consists of an uncountable  set $Y$ together with the Borel sets given by a Polish topology on $Y$.  Given an  uncountable Borel subset $B$ of a Polish space, its Borel subsets  induce a standard Borel space on $B$; in particular, the set $\+ B$ above carries the structure of such a space.

\begin{definition}  \label{def:CWB}
Equivalence relations $E$ and $F$ on standard Borel   spaces $X$ and $Y$, respectively, are called \emph{classwise Borel bireducible} if there are Borel reductions $F\colon X\rightarrow Y$ of $E$ to $F$ and $G\colon Y \rightarrow X$ of $F$ to $E$ such that their factorings $\hat{F}\colon X/E \rightarrow Y/F$ and $\hat{G}\colon Y/F \rightarrow X/E$ to the quotient spaces are bijections and satisfy $\hat{F}=\hat{G}^{-1}$. 
\end{definition}

A main motivation for  this notion comes from duality theorems, such as Stone's. By a Stone space we mean a compact,  totally disconnected  topological  space. Stone duality sets up an correspondence between such spaces and    Boolean algebras. This restricts to a duality between separable Stone spaces and countable Boolean algebras.
 
 While a duality theorem merely requires that  the objects are  given up to isomorphism, here we are interested in  Borel versions.  We   assume that the objects  are concretely given as points in standard Borel spaces.  For instance, a  Borel version of   Stone duality is as follows:   we   assume that an  infinite Stone spaces  is    given as the set of paths $[T]$ through a subtree $T$ of $\strcantor$ with infinitely many paths and without dead ends, and that the Boolean algebras have domain $\omega$.
 Abstractly, for an infinite Stone space $G$, one  lets $\+M(G)$ be the Boolean algebra of clopen sets in $G$; for each countably infinite  Boolean algebra $M$, one lets $\+ G(M)$ be the  Stone space  of ultrafilters of $M$.  If the objects are concretely represented, these maps become Borel: 
  given  a tree~$T$ such that $G= [T]$, one  can in a Borel fashion produce a listing without repetition of the clopen sets of  the space~$[T]$, and hence determine $\+ M(T)$. Given a Boolean algebra~$M$ with domain $\omega$, one  lets $\+G(M)$ be the tree  of strings $\sss \in \strcantor$ describing a nonzero conjunction $r_\sss$ of literals (for $i<\sssl$ let $l_i = i $ if $\sss(i)=1$ and $l_i= \lnot i$ otherwise; let  $r_\sss= \bigwedge_{i<n} l_i$). Clearly    the conditions above for classwise Borel bireducibility hold.
  
  Our Borel version of duality between classes $\+ A, \+ B$ of (possibly topological) structures requires that  both $\+ A$ and $\+ B$ are isomorphism invariant Borel sets of the canonical Polish spaces for such structures. In Section~\ref{ss:SBB compact} we   obtain a duality between  the class $\+ A$ of profinite groups and an appropriate class $\+ B$ coarse groups.  The result of Kechris, Nies and Tent \cite[Section 3.3]{Kechris.Nies.etal:18} (also see Prop.\ \ref{fact:standard}) on Roelcke precompact groups mentioned above is not yet such a duality result, because   one   needs to show that the closure under isomorphism of the range of the operator $\+ M$ defined on such groups is Borel, and also needs to define a Borel inverse $\+ G$ up to isomorphism. This is indeed possible for the Roelcke precompact groups, and will be carried out in a forthcoming paper of Melnikov and the first author. It is not known whether  the larger class of closed subgroups of $\S$ that have only countable many open subgroups is Borel, so at this stage one  can't expect a Borel duality result here.

 We note that  a common way of formalizing abstract duality theorems is via the  notion of  equivalence of categories, where the operators (now functors) also turn isomorphisms into isomorphisms. A Borel version of category equivalence can be obtained for the profinite, and even the Roelcke precompact, groups, but is   unknown for the oligomorphic groups. The construction in Section~\ref{s: turning} of the inverse map  $\+ G$ is not uniform in that sense, because one has to pick an appropriate element $W$ of the coarse group in a Borel fashion (see Section~\ref{ss:upperbound} below). This  choice   is not necessarily unique. If one also wanted to  turn  any isomorphism between coarse groups into an isomorphism of the corresponding oligomorphic groups,   one would need to make choices of elements in the coarse groups that are matched by the given isomorphism, which is only possible if the elements are uniquely determined.   In contrast, the weaker formalization of Borel duality given by Definition~\ref{def:CWB} works.
 \subsection{The upper bound on the complexity of isomorphism} 
  \label{ss:upperbound}  \label{s:bi}
 Once Theorem~\ref{thm:main} is established,  we will  show that isomorphism of oligomorphic groups is Borel reducible to a countable Borel equivalence relation. We apply   a result of Hjorth and Kechris \cite[Theorem 4.3]{Hjorth.Kechris:95} about  Borel invariant classes~$\+ C$ of countable structures that will be explained in more detail at the beginning of Subsection~\ref{ss:EC}.  
 An important further ingredient (just alluded to above) is that  each   oligomorphic group  $G$ has an open subgroup $W$ such that the left translation action of  $G$ on the left cosets of $W$ is  oligomorphic, and yields a topological embedding of $G$ into $\S$. ($W$ is simply the intersection of the stabilisers of  finitely many numbers chosen to represent the 1-orbits; see Lemma \ref{canonical oligomorphic action}.)
   We   thank Todor Tsankov for communicating this fact to us.  
%

We mention here that there is  an  alternative way  to obtain  the upper bound on isomorphism of oligomorphic groups  from Theorem~\ref{thm:main}: via  bi-interpretability of $\omega$-cate\-gorical structures. 
 To an oligomorphic group  $G$ one can in a Borel  way assign a structure $N_G$ with domain $\omega$ such that $G= \Aut(N_G)$: the language has $k_n$ many $n$-ary relation symbols $P^n_i$, where $k_n$ for $n\ge 1$ is the number of $n$-orbits of $G$, and $P^n_i$ denotes in $N_G$  the $i$-th $n$-orbit.  Coquand (unpublished), see Ahlbrandt and Ziegler \cite{Ahlbrandt.Ziegler:86},   showed  that oligomorphic groups $G$,  $H$ are topologically isomorphic if and only if $N_G$ and $N_H$ are bi-interpretable in the sense of model theory (e.g.\ Hodges \cite[Section 5.3]{Hodges:93}); also see    \href{http://wwwf.imperial.ac.uk/~dmevans/Bonn2013_DE.pdf}{David Evans' 2013  notes}.
 
 One can show that bi-interpretability of $\omega$-categorical structures   is a  $\mathbf \Sigma^0_2$ relation. Now one   applies  a related result of Hjorth and Kechris in the same paper \cite[Theorem 3.8]{Hjorth.Kechris:95},  by which the existence of a Borel reduction of $\cong_\+ B$  to a $\mathbf \Sigma^0_2$ equivalence relation implies that $\cong_\+ B $ is essentially countable. 

 We don't follow this pathway because the formal details   would be very  tedious, while after our proof of  Theorem~\ref{thm:main}   not too much extra effort  is required to  satisfy  the hypothesis of \cite[Theorem 4.3]{Hjorth.Kechris:95}. For some details on the alternative  approach see our Logic Blog post~\cite[Section 8.5]{LogicBlog:18}.
 
%

\subsection{Preliminary: the Effros space}
  \label{s:prelim}


Given a Polish space $X$, let $\+ E(X)$ denote  the set of closed subsets of $X$.   The \emph{Effros Borel space} on $X$  is the standard Borel space consisting of  $\+ E(X)$  together with  the $\sigma$-algebra   generated by the sets \bc $\+ C_U = \{ D \in \+ E(X) \colon D \cap U \neq \ES\}$,  \ec for open $U \sub X$.  For details, see e.g.\ \cite[Definition 1.4.5]{Gao:09}.

It is not hard to see that in $\+ E(\S)$, the properties of being a  (closed) subgroup of $\S$, and of being an oligomorphic group, are  Borel. For the former see \cite[Lemma 2.5]{Kechris.Nies.etal:18}. For the latter, note that a closed subgroup $G$ is oligomorphic if and only if for each $n$, there is $k$  such that 

\bc $\ex x_1, \ldots, x_k \in T_n \fa y \in T_n   \bigvee_{1 \le i \le n } G \cap U_{x_i, y }   \neq \ES$, \ec
where $T_n$ is the set of $n$-tuples of natural numbers without repetitions, and $U_{x,y}$ for $x, y \in T^n$ is the open set of permutations $f$ such that $f(x(r) ) = y(r) $ for each $r < n$.





\section{Coarse groups}
  
\label{s:coarse}
In this section we study coarse groups, defined in Section~\ref{s:CG} above, and  in particular isolate some of their  properties by formulating  axioms. This 
leads to 
an abstract axiomatisation 
of coarse groups of interesting classes of groups. 
We let $M$ always denote a coarse group. 
In our applications   we will only consider the case that $M$ is  countable.   We introduce the filter group $\+ F(M)$ as a step towards recovering closed subgroups of $\S$ from their approximative countable coarse groups.

\subsection{Basic definitions and axioms}  
 
 Throughout, $M$ will denote a structure for the signature with a ternary relation symbol $R$ which will describe a coarse group. We now give the axioms for such a   structure,  and we will always assume that $M$ satisfies them.

For $A, B, C \in M$, the ternary relation $R(A,B,C)$ will more suggestively be written   as ``$A B\sqsubseteq C$". Since we think of elements of $M$ as cosets of open subgroups, 
we use group theoretic terms marked with an asterisk.   Thus, for instance, we will refer to the elements of $M$ as $^*$cosets or $^*$subgroups. Letters $A, B, C, \ldots $ denote general elements of $M$.  

\begin{definition}[Some definable relations]  
\label{df:main}  \mbox{}
 \begin{enumerate} \item[(a)] $A \in M$ is a $^*$subgroup  if $AA \sqsubseteq A$. Letters $U,V,W $ denote $^*$subgroups in $M$. 
We write $U \sqsubseteq V$ for $UV \sub V$. 

\item[(b)]   
 $A$ is a left $^*$coset of a $^*$subgroup $U$ if and only if $U$ is the 
largest subgroup under $\sqsubseteq$ with $AU \sqsubseteq A$. ($U$ is 
unique by definition.) 
We define right $^*$cosets  of $U$ analogously. We write $LC(U)$ for the set of left $^*$cosets and $RC(U)$ for the set of right $^*$cosets of a $^*$subgroup $U$. 

\item[(c)] We write $A\sqsubseteq  B$ for arbitrary   $A,B$ if  $AU \sqsubseteq B$ for some $^*$subgroup $U$ such that  $A$ is a left $^*$coset for $U$. ($U$ exists, and is unique, by  Axiom \ref{axiom basics}(b).) 
\end{enumerate}
\end{definition} 

\setcounter{axiom}{-1} 
\begin{axiom} (Basic axioms)  \label{axiom basics}
\label{basic axioms} 
 \begin{enumerate}
 \item[(a)] The relation $\sqsubseteq$ on $^*$subgroups is a partial order so that any two elements have a meet.

\item[(b)]   Every element  is   a left $^*$coset of some $^*$subgroup and a right $^*$coset of  some $^*$subgroup.
 
\item[(c)]   The  relation  $\sqsubseteq$  in Definition~\ref{df:main}(c) is a partial order that extends $\sqsubseteq$ on the set of $^*$subgroups. 
 \end{enumerate}
\end{axiom}

\begin{axiom}  \label{ax:order} (Monotonicity) 

\n  
 If $B_0 B_1\sqsubseteq C$ and $A_i\sqsubseteq B_i$ for $i\leq1$, then $A_0 A_1\sqsubseteq C$. 
\end{axiom}

  We say that~$A$,$B$ in $M$ are   \emph{disjoint} if $\neg \exists C\, (C\sqsubseteq A \wedge C\sqsubseteq B)$.\footnote{Disjointness can be expressed via the ternary relation $\sqsubseteq$ as $\neg \exists C \exists D\,   (C D\sqsubseteq A \wedge CD\sqsubseteq B)$; this follows from Axioms \ref{axiom basics}, \ref{ax:order}, \ref{axiom disjointness of cosets} and \ref{axiom products of left and right cosets}.} 

\begin{axiom} \label{axiom disjointness of cosets} Suppose that  $U' \sqsubseteq U$ and 
$A'\in LC(U')$. 

\n (a) There is $A \in LC(U)$ such that $A' \sqsubseteq A$.

\n  (b) If  $A\in LC(U)$, 
then $A' \sqsubseteq A$,  or $A'$ and $ A$ are disjoint.  In particular, any two distinct  left $^*$cosets  of the same $^*$subgroup    are disjoint.   

Similar statements  holds for right $^*$cosets. 
\end{axiom}

\begin{remark} \label{rem:defs} {\rm  \cite[Section 3.3]{Kechris.Nies.etal:18}.  For a closed subgroup  $G$ of $\S$,  the terms introduced above have their intended meanings in the structure $\+ M(G)$, and the axioms specified so far are satisfied. }
\end{remark}

\begin{axiom} \label{axiom cosets for smaller subgroups} 
Let  $U$ and $V$ be  $^*$subgroups and $B\in LC(V)$.  Then 
 \bc $U\sqsubseteq V$ $\LR$
 there is $A\in LC(U)$ with $A\sqsubseteq B$. \ec
A similar statement holds for right $^*$cosets. 
\end{axiom} 
\n To see that this axiom holds in $\+ M(G)$, let $B=bV$. If $U\sqsubseteq V$, then $A=bU\sqsubseteq bV$. Conversely, suppose that  $A=aU \sqsubseteq B$. Then   $b^{-1}aU\sqsubseteq V$. So $a^{-1}b\in V$, and hence  $U\sqsubseteq a^{-1}bV=V$.  

\smallskip
Each $A\in M$ contains a left  $^*$coset of an arbitrarily ``small" $^*$subgroup $V$. 

\begin{claim} \label{axiom existence of full filters} 
\

\n For each   $A\in M$ and each $^*$subgroup $U$, there are  a $^*$subgroup $V\sqsubseteq U$ 
and a left $^*$coset $B$ of $V$ 
such that $B\sqsubseteq A$. A similar fact holds for right $^*$cosets. 
\end{claim} 
 \n To see this, suppose $A \in LC(W)$. Let $V= W \wedge U$ using Axiom \ref{axiom basics}(a). By Axiom~\ref{axiom cosets for smaller subgroups} there is $B\sqsubseteq A$  in $LC(V)$, as required.

\smallskip

We write $\+ S(A,B)$ for the statement that there is a $^*$subgroup $V$ such that 
$A\in RC(V)$, $B\in LC(V)$, 
and $AB \sqsubseteq V$. It is easily checked that in $\+ M(G)$, we have $\+ S(A,B) \lra \+ S(B,A) \lra B = A^{-1}$. In particular, if we are given $A\in\+ M(G)$, then $B\in\+ M(G)$ is unique.  

\begin{axiom}[Inverses] \mbox{} \label{axiom inverses 1}
\begin{enumerate-(a)} 
\item 
For each $A$, there is a unique $B$ such that $\+ S(A,B)$.
\item 
$ \+ S(A,B) \lra \+ S(B,A)$.  Assuming the axiom holds in a structure $M$, instead of $\+ S(A,B)$ we will    write $B = A^\diamond$.

\item $A \mapsto {A}^\diamond$ is an isomorphism with respect to $\sqsubseteq$. 
\end{enumerate-(a)} 
\end{axiom}
Note that Axiom~\ref{axiom inverses 1} implies  that $A^{\diamond \diamond}=A$.

\subsection{Full filters, and the filter group}  

A subset $x$ of $M$ is called \emph{closed upwards  with respect to $\sqsubseteq$} if for all $A\in x$ and $A\sqsubseteq B$, we have $B\in x$. It is \emph{directed downward} if for all $B,C\in x$, there is some $A\in x$ with $A\sqsubseteq B$ and $A\sqsubseteq C$. 
We now define the set of full filters $\F(M)$. Thereafter we will  define a group operation and add axioms   ensuring that  $\F(M)$ with a canonical  topology forms  a Polish group.

\begin{definition}[Full filters] \label{df:filter}  \mbox{} \\
 A \emph{full filter $x$ on $M$} is a subset of $M$ with the following properties.  
\begin{enumerate-(a)} 
\item 
It is directed downwards, and closed upwards  with respect to $\sqsubseteq$. 
\item 
 Each $^*$subgroup $U$ in $M$ has a left $^*$coset and a right $^*$coset in $x$. 
\end{enumerate-(a)} 
 
\n  We let $\F(M)$ denote the set of full   filters on $M$.  

\n Letters  $x$, $y$, $z$ will denote  elements of $\F(M)$. 
\end{definition} 
We note that full filters are separating sets for the equivalence relation on $M$ of being left $^*$cosets of the same $^*$subgroup, and similarly for right $^*$cosets. In particular, they are maximal filters.
 \begin{claim} \label{build full filter} Suppose that   $M$ is countable. For each  $A\in M$ there is a full filter~$x$ such that $A \in x$.\end{claim} 
\n To see this, let $\seq {U_n}_{n\in\omega}$ be a listing of  all $^*$subgroups in $M$. We construct a $\sqsubseteq$-decreasing sequence  $\seq{A_n}\sN n$ as follows. Let $A_0=A$. Given $A_n$, find $V_n\sqsubseteq U_n$ and a left $^*$coset $B_n$ of $V_n$ with $B_n\sqsubseteq A_n$ by Claim~\ref{axiom existence of full filters}. Similarly, take $W_n\sqsubseteq V_n$ and a right $^*$coset $A_{n+1}$ of $W_n$ with $A_{n+1}\sqsubseteq B_n$. Then $\{C\colon \exists n\ A_n\sqsubseteq C\}$ is a full filter on $M$ containing $A$. 

\begin{definition}[Topology on the set of  full filters]  \label{def:topo}  \  \\ We define a  topology on $\F(M)$ by declaring  as subbasic        the  open sets 
 
\bc $\hat{A}=\{x\in \F(M)\colon A\in x\}$ \ec where  $A\in M$.  These sets    form a base since  filters are directed.  \end{definition}

The Baire space ${}^\omega \omega$ is endowed with  a topology given by the subbasic open sets $\{ f \colon \, f(n)= i\}$ for $n, i \in \omega$. 
It is \emph{totally disconnected}, i.e. any connected subset contains at most one element. 

%

\begin{prop} \label{filter space is Polish}  Suppose that $M$ is countable. Then 
$\F(M)$ is a totally disconnected Polish space. 
\end{prop} 
\begin{proof} Since $M$  is countable, the $^*$subgroups and $^*$cosets   in $M$ can be   provided with an ordering of type $\omega$. 
Let  $U_n$ denote  the $n$-th $^*$subgroup in~$M$. 

We define an  injection $\Delta$ from $\F(M)$ into   Baire space ${}^\omega \omega$.  
  Suppose that $x\in \F(M)$. Let $\Delta(x)(2n)$ be the unique $i$ such that the $i$-th left $^*$coset of $U_n$ in $M$ is an element of $x$. Let   $\Delta(x)(2n+1)$ be the unique $i$ such that the $i$-th right $^*$coset of $U_n$ in $M$ is an element of $x$. By Axiom \ref{axiom disjointness of cosets}, $\Delta$ is well-defined and injective.

We claim that $\Delta$ is a homeomorphism to $\ran(\Delta)$. It is clear that the image of any basic open subset $\hat A$ of $\+ F(M)$ is open. Conversely, the preimage of any subbasic open subset of the Baire space is of the form $\{x\in \F(M)\mid  A \in x\}$ for some $A \in M$.


%
%

We  show that $\ran (\Delta)$ is $G_\delta$.  Identifying a full filter $x$ with $\Delta(x)$, upwards closure is a closed condition. The condition in Def.\ \ref{df:filter}(b) is satisfied automatically, because  $f(2n)$ denotes  a left $^*$coset of $U_n$, and $f(2n+1)$  a right $^*$coset of $U_n$.  Downwards directedness is a $G_\delta$ condition.

  Hence $\F(M)$ is homeomorphic to a $G_\delta$ subset of the Baire space. Since every $G_\delta$ subspace of a Polish space is again Polish, it follows that $\F(M)$ is a Polish space.  Clearly it is totally disconnected since Baire space is.
\end{proof} 

%
%


%

 For $x \in \+ F(M)$ we  let \[x^{-1}=\{A^\diamond\mid A\in x\}.\] We claim that $x^{-1}$ is   a full filter. It is upwards closed and directed by the previous axiom. The condition Def.\ \ref{df:filter}(b) holds since the $*$ operation interchanges   left $^*$cosets of  $U$ with  right $^*$cosets of $U$. 
Since $A^{\diamond \diamond}=A$, we further have $\hat{A}^{-1}=\hat{A^\diamond}$. 

\begin{definition}[Product of full filters]   For full filters $x,y$ on $M$, we put
 \[x\cdot y=\{C\in M\mid \exists A\in x \exists B\in y\ AB\sqsubseteq C\}.\] 
\end{definition} 

 The next axiom ensures  that this   operation on $\F(M)$  is defined and continuous.  
\begin{axiom} \label{axiom products of left and right cosets} 
Suppose $A\in RC(U)\cap LC(V)$ and $B\in  RC(V)\cap LC(W)$.

\n There is $C\in RC(U)\cap LC(W)$ such that  $A B\sqsubseteq C$, and  $C$ is the least   $D$ with $A B\sqsubseteq D$.  We will write $A\cdot B$ for this  (unique)   $C$.
\end{axiom}

This holds in  $\+ M(G)$: if $A=aV=Ua'$ and $B=bW=Vb'$, then $AB=aVVb'=aVb'$. Since $aVb'=Ua'b'=abW$, $C=AB$ is a right coset of $U$ and a left coset of $W$. 
It is clearly minimal.

\begin{claim} \label{product of filters is a filter} 
$x\cdot y$ is an element of $\F(M)$ for all $x,y\in \F(M)$. 
\end{claim} 
\begin{proof} 
Since the relation $\sqsubseteq$ is transitive, $x\cdot y$ is  closed upwards by definition of the product. 

To see that $x\cdot y$ is  directed downwards, suppose that elements $C_0, C_1\in x\cdot y$ are given. Then there are  $A_0,A_1\in x$ and $B_0,B_1\in y$ such that  $A_i B_i\sqsubseteq C_i$ for $i= 0,1$. Since $x$ and $y$ are directed downwards, there are $A\in x$ with $A\sqsubseteq A_0$ and  $A\sqsubseteq A_1$, and $B\in y$ with $B\sqsubseteq B_0$ and $B\sqsubseteq B_1$. By monotonicity, $AB\sqsubseteq C_i$ for $i\leq1$. 
Let  $V$, $W$  be the $^*$subgroups in $M$ such that  $B\in  RC(V) \cap LC(W)$. 

We can assume $A\in LC(V)$ 
by shrinking $A$, $B$, $V$, $W$ while using the fact that $x$ and $y$ are full filters to maintain that $A\in x$ and $B\in y$. 
In more detail, take any $A\in LC(U)$. 
Let  $V' =  U\wedge V$ by Axiom \ref{basic axioms}(a) and let $A'\in LC(V')\cap x$ using that $x$ is a full filter. 
$A$, $A'$ cannot be disjoint, since $A,A'\in x$ and $x$ is a filter, so $A'\sqsubseteq A$ by Axiom \ref{axiom disjointness of cosets}. 
One similarly obtains $B'\sqsubseteq B$ in $RC(V')\cap y$. 
We have $W'\sqsubseteq W$  for the unique *subgroup $W'$ with $B'\in LC(W')$, by Axiom \ref{axiom cosets for smaller subgroups}. 
Thus $A'$, $B'$, $V'$, $W'$ are as required. 

By Axiom \ref{axiom products of left and right cosets}, there is a unique $C\in LC(W)$ 
with $AB\sqsubseteq C$ and by its minimality, $C \sqsubseteq C_i$ for $i\leq1$. Since $AB\sqsubseteq C$, we have $C\in x\cdot y$. 

We now show that $x\cdot y$ satisfies condition (b) in Def.\ \ref{df:filter}. Take any $^*$subgroup $W$ in $M$ and let  $B\in LC(W)\cap  y$. Let $V$ the $^*$subgroup such that $B \in RC(V)$, and let $A \in LC(V) \cap x$.
By Axiom \ref{axiom products of left and right cosets}, there is $C\in LC(W)$ 
with $AB\sqsubseteq C$. Then $C\in x\cdot y$.  Right $^*$cosets are similar.
\end{proof} 

The next  axiom can be expressed by a $\Pi^1_1$ condition in case that $M$ is countable. 
An equivalent first-order axiom (Axiom~\ref{axiom associativity2}) is introduced in Subsection~\ref{ss:replace fo}. 
We work with the simpler $\Pi^1_1$ axiom here, since it suffices for the main results. 


\begin{axiom} \label{axiom product of filters}  \label{axiom associativity} 
 The   operation $\cdot$   on $\F(M)$ is associative. 
\end{axiom} 

\begin{remark}[Neutral element of $\F(M)$] \label{rem:neutral} Let $1_{\F(M)}$ denote the filter generated by the $^*$subgroups in $M$. We   have $1_{\F(M)}\cdot x=x\cdot 1_{\F(M)}=x$ and $x\cdot x^{-1}=x^{-1}\cdot x=1_{\F(M)}$ for all $x\in\F(M)$ by Axioms \ref{axiom inverses 1} and \ref{axiom product of filters}.  \end{remark}

%

The next axiom ensures that $AB\sqsubseteq C$ and $A\sqsubseteq B$ express the expected  properties in $\F(M)$. 
It holds in $\+ M(G)$ by continuity of the group operation. 

\begin{axiom} \mbox{} \label{axiom negative expression of subset} 
\begin{enumerate-(a)}
\item 
$AB\sqsubseteq C$ iff there are no $^*$cosets $D\sqsubseteq A$, $E\sqsubseteq B$ and $F$ such that  $DE\sqsubseteq F$ and $C$,$F$ are  {disjoint}. 
\item 
$A\sqsubseteq B $ iff there is no $^*$coset $C\sqsubseteq A$ with $B$, $C$ disjoint. 
\end{enumerate-(a)} 
\end{axiom} 

Recall from Definition~\ref{def:topo} that $\hat{A}=\{x\in \F(M)\colon A\in x\}$. Note that if $A, B$ are disjoint (as defined before Axiom~\ref{axiom disjointness of cosets}) then $\hat A \cap \hat B= \ES$, because filters are  directed downwards. 
Let $\hat{A}\hat{B}$ denote the setwise product. 

\begin{claim} \label{subset is correct}  Let  $A,B,C \in M$.  Let $U$ be a $^*$subgroup. 
\bi \item[(a)]
$AB\sqsubseteq C \Longleftrightarrow \hat{A}\hat{B}\subseteq \hat{C}$.
\item[(b)]  $\hat U$ is a subgroup of $\F(M)$.
\item[(c)]  $A\sqsubseteq B \Longleftrightarrow \hat{A}\subseteq \hat{B}$. 

\item[(d)] $\hat {B^\diamond} = (\hat{B})^{-1}$. 
\item[(e)]  $A\in LC(U) \Longleftrightarrow\hat A $ is a left coset of $\hat U$, and similar for right cosets.

 \ei
\end{claim} 
\begin{proof} 
(a) 
The forward implication is clear. For the converse implication, we assume that $AB\not\sqsubseteq C$ and   find full filters $x\in\hat{A}$ and $y\in \hat{B}$ with $x y\notin\hat{C}$. 
By Axiom \ref{axiom negative expression of subset}(a), there are $D\sqsubseteq A$, $E\sqsubseteq B$ and $F$ with $DE\sqsubseteq F$ and $C$, $
F$  {disjoint}. By  Claim~\ref{build full filter}, take full filters $x\in \hat D$ and $y\in\hat E$. Then $x y\in \hat{F}$. Since   $\hat{C}\cap \hat{F}= \ES$ we conclude that  $x y\notin\hat{C}$.

\n (c)   is similar to (a) using Axiom \ref{axiom negative expression of subset}(b), and (d) is easily verified.

\n  (e) forward implication: take any $x\in \hat{A}$. 
We  show that $x\hat{U}=\hat{A}$.

For $x\hat{U}\subseteq \hat{A}$, let $y\in \hat{U}$. Since $A$ is a left $^*$coset of $U$, we have $AU\sqsubseteq A$. So $x\cdot y\in \hat{A}\hat{U}\subseteq\hat{A}$ by (a).

For  $\hat{A}\subseteq x\hat{U}$, let $y\in \hat{A}$. To show  that $y\in x\hat{U}$, or equivalently $x^{-1} y\in \hat{U}$,  note that we have $x^{-1}y\in \hat{A}^{-1}\hat{A}=\hat{A^\diamond}\hat{A} \subseteq \hat{U}$ by (d) and (a). 

\n (e) backward implication: Suppose $\hat A = x \hat V$. There is $B \in x$ such that $B \in LC(V)$. By the forward implication, $\hat B$ is a left coset of $\hat V$. Also $x \in \hat A \cap \hat B$, so $A, B$ are not disjoint.  Since $A, B \in LC(V)$ this implies $A=B$ by Axiom~\ref{axiom disjointness of cosets}. Right cosets are dealt with symmetrically.
\end{proof} 




%
%

We can't   prove on the basis of the present axioms that $(\F(M),\cdot)$ is isomorphic to a closed subgroup of $\S$; 
this will be achieved in Section \ref{s: turning} in the oligomorphic case, and in Section~\ref{ss:SBB compact} in the profinite case. 
At the current point  we can show the following.  
\begin{prop} \label{filter group is Polish} \ 

\n  Suppose that  $M$ is countable. Then 
$(\F(M),\cdot)$ is a Polish~group. 
\end{prop} 
\begin{proof} 
In view of Prop.~\ref{filter space is Polish}, it suffices  to show that the   operation  $x,y \mapsto x\cdot y^{-1}$  on $\F(M)$ is continuous. I.e. it suffices  to show that for all $x$, $y$ and every  $D \in M$ with $x\cdot y^{-1}\in \hat{D}$, there are  $A, B\in  M $ with $A\in x$, $B\in y$ such that $u\cdot v^{-1} \in \hat{D}$ holds for all $u\in \hat{A}$ and $v\in\hat{B}$. 

To see this, suppose that $D\in x\cdot y^{-1}$ is a left coset of $U$. We choose a left coset $B\in y$ of $U$. 
By Axiom~\ref{axiom basics}, $B$ is a right coset of some $V$. We choose a left coset $A\in x$ of $U$. 
By Axiom~\ref{axiom products of left and right cosets}, there is a left coset $C$ of $V$ with $AB^\diamond\sqsubseteq C$ and hence $C\in x\cdot y^{-1}$. Since $x\cdot y^{-1}$ is a full filter, we have $C=D$. 
By Claim \ref{subset is correct}, $\hat{A}\hat{B^\diamond}= \hat{A}(\hat{B})^{-1}\subseteq \hat{C}=\hat{D}$. 
\end{proof}

%
%
%

We   call $\F(M)$ the \emph{filter group} of $M$. Note that $\+ F(M)$ is defined abstractly as a Polish group, rather than  as a permutation group. 
%
We show that 
closed subgroups of $\S$ with countably many open subgroups can be recovered  in a canonical way as the filter group of  their coarse group.
\begin{prop}[cf.\ ~\cite{Kechris.Nies.etal:18}, after Claim 3.6] \label{fact:standard} \

\n Suppose that $G $ is a closed subgroup of $ \S$ such that  $\+ M(G)$ is countable. There is a natural group homeomorphism  \bc $\Phi: G \cong \+ F( \+ M(G))$ given  by   $g \mapsto \{ A \colon A \ni g\}$, \ec with inverse  $\Xi$ given by $x \mapsto g$ where $ \bigcap x = \{g\}$.    \end{prop}
This is essentially contained 
in~\cite{Kechris.Nies.etal:18} (note that $L_g= R_g= \{ A \colon A \ni g\}$ in the notation there). 
\begin{proof}[Sketch of proof.]
Let $x \in  \+ F( \+ M(G))$. We  show  that $\bigcap x$ is non-empty. 

Let $U_n$ be the open subgroup of $G$ consisting of the permutations that fix $0, \ldots, n$. Since $x$ is a full filter, there are permutations $r_n, s_n \in G$ such that $r_n U_n \in x$ and $U_n s_n \in x$. Let $g(n) = r_n(n)$ and $g^*(n)= s_n^{-1}(n)$. As in~\cite{Kechris.Nies.etal:18} one shows that $g^*= g^{-1}$ using that $x$ is a filter. So $g$ is a permutation, and then clearly $g\in \cap x$ since $G$ is closed.   

On the other hand, since the open cosets form a base,   $\bigcap x$ has at most one element. So the   map $\Xi$  is defined.

Let $g \in G$ and let $x \in \+ F(\+ M(G))$.   Trivially $\Xi(\Phi(g)) = g$. 
It is also trivial that  $x \sub y= \Phi(\Xi(x))$. Since $y$ is a full filter this implies $x=y$.

One shows that $\Phi$  preserves the group operations  as in~\cite[after Claim~3.6]{Kechris.Nies.etal:18}. 
Finally  $\Phi^{-1}(\hat A) =   A$  by definition, so   $\Phi$ is a homeomorphism.
\end{proof}

\begin{remark} 
Note that by this argument, the group~$\Aut(G)$ of topological automorphisms of $G$ is naturally isomorphic to $\Aut(\+ M(G))$. Hence $\Aut(G)$ can itself be seen as a closed subgroup of $\S$.
\end{remark} 

\begin{remark} \label{rem:normal} The following illustrates the structure of  a coarse group~$M$ and will be useful when we discuss  coarse groups of profinite groups in Subsection~\ref{ss:profinite}. Recall that a subgroup of a group  is normal if its left cosets coincide with the right cosets. Thus,  we say that a $^*$subgroup $U$ of $M$ is \emph{normal} if $LC(U)=RC(U)$. 

We verify  that for normal $U$, the set  $LC(U)$    carries a canonical group structure. (If $M= \+ M(G)$, then  this group is simply  $G/U$.)   
Note that by Axiom~\ref{axiom products of left and right cosets}, if $A,B \in LC(U)$ then there is a unique $C \in LC(U)$ such that $AB \sqsubseteq C$. (Recall that we write $A\cdot B=C$.) Also,  if $A \in LC(U)$ then by definition $A^\diamond \in RC(U)= LC(U)$. By associativity and the definition of the operation~$\diamond$, $(LC(U), \diamond)$ is a group. 
\end{remark}

\begin{remark} We sketch an alternative   approach to coarse groups which may be useful  e.g.\  in future applications to totally disconnected locally compact (t.d.l.c.) groups.  On the set of open cosets, or the compact open cosets in the t.d.l.c.\ case,  we    take as primitives the  following notions:    the inclusion partial order $\sq$, $^*$subgroups, the two operations mapping a $^*$coset $A$ to  the $^*$subgroups $U,V$ such that $A \in RC(U) $ and $A \in LC(V)$, and  the product $A\cdot B$ of cosets as in Axiom~\ref{axiom products of left and right cosets}, namely, under the assumption  that $A \in LC(V)$ and $B \in RC(V)$ for the same $^*$subgroup~$V$. 

Recall that a groupoid is a small category in which every morphism has an inverse. An \emph{inductive groupoid}  is a groupoid that also has the structure of a meet semilattice, and satisfies some natural axioms positing that  the two structures are compatible~\cite[Section 4.1]{Lawson:98}. In this alternative setting,  a  coarse group   has the structure of an   inductive groupoid, where $A$ as above is seen as a morphism $U \to V$.

The two approaches to coarse groups are equivalent. Each coarse group as developed above can be first-order defined in an inductive groupoid satisfying appropriate additional axioms, and vice versa. It is clear that an inductive groupoid can be defined in the coarse group.   Conversely, suppose we are given an inductive groupoid~$L$. If  $A \in LC(U) $ and $B \in RC(V)$, letting $W= U \wedge V$, we can define  \[ AB \sq C :\LR \ex A' \sq A \ex B' \sq B [ A'\in LC(W) \land B' \in RC(W) \land A' \cdot B' \sq C].\] 
The axioms above up to Axiom~\ref{axiom products of left and right cosets}, but with the exception of Axiom~\ref{axiom disjointness of cosets}, turn into the usual axioms for inductive groupoids OG1-OG3 as in~\cite[Section 4.1]{Lawson:98}. 
 Conversely, formulating appropriate additional axioms for $L$, one can obtain all the axioms we provide. For some  details see \cite[Part~1]{LogicBlog:20}. \end{remark}
\begin{example} 
As  an instructive  example of an inductive groupoid we consider the oligomorphic group~$G= \Aut(\mathbb Q, <)$. The open subgroups of $G$ are the stabilizers of finite sets. If $U,V$ are stabilizers of sets of the same finite cardinality, there is a unique morphism $A \colon U \to V$ in the sense above, corresponding to the order-preserving bijection between the two sets. The inductive groupoid for $\Aut(\mathbb Q, <)$ is   canonically isomorphic to the groupoid of finite order-preserving maps on $\mathbb Q$, with the partial order being   reverse extension. 

A filter $x$ (in either setting) corresponds  to an arbitrary  order-preserving map $\psi$ on $\mathbb Q$. The filter $x$ contains a right coset of each open subgroup if and only if $\psi$ is total, and a left coset of each open subgroup if and only if  $\psi $ is onto. So the set of full filters corresponds to $\Aut (\mathbb Q)$ as expected. (Incidentally, this example shows that in Definition~\ref{df:filter}(b)  we  need both sides, and that not every maximal filter is full.) \end{example}

\subsection{The action of $\+ F(M)$ on $LC(V)$} \

\n Suppose that $V$ is a $^*$subgroup of $M$, and as before let $LC(V)\sub M$ denote its set of left $^*$cosets. We define an action 
\begin{equation} \label{eqn:gamma}\gamma_V\colon \F(M)\curvearrowright LC(V)\end{equation}
by letting \begin{equation} \label{eqn:action} x\cdot A= B \text{ iff } \exists  S\in x   \, [ SA\sqsubseteq B]. \end{equation}  
Note that such a $B$ is unique because $x$ is a filter. For $M=\+ M(G)$, by Prop.~\ref{fact:standard}, $\gamma_V$ is simply the natural left   action of $G$ on the left cosets of the open subgroup $V$.  

We verify  that $\gamma_V$ is defined:   For each full filter $x$, for every $^*$subgroup~$V$ and  each left   $A\in LC(V)$, there is   $B \in LC(V)$  such that $x \cdot A =B$.  
Suppose that    $A\in RC(U)$. 
Since $x$ is a full filter, it contains some $S\in LC(U)$. 
Then $B:=S\cdot A\in LC(V)$ by Axiom \ref{axiom products of left and right cosets}. Hence $x\cdot A=B$.


\begin{claim} \label{cl:BoJo} For each full filter $x$, for each $C \in M$, we have $x \hat {C}= \hat {x\cdot C}$.  \end{claim}

\begin{proof} Suppose $C \in LC(V)$. By Claim~\ref{subset is correct}(e),  both    $ x \hat C$ and $\hat {x \cdot C}$ are left cosets of $\hat V$. So it suffices to show that 
$x \hat {C} \sub  \hat {x  \cdot  C}$. Let $D = x \cdot C \in LC(V)$. By definition there is $S \in x$ such that $SC \sqsubseteq D$.  If $z \in x \hat C$, then by definition  there is $y$ such that $C \in y$ and $z = x\cdot y$.  So $D \in z$ as required.\end{proof}
\begin{claim} \label{axiom gammaV is a group action} \

\n     For every $^*$subgroup~$V$, $\gamma_V\colon \F(M)\curvearrowright LC(V)$ is a group action. 
\end{claim} 
\begin{proof} Clearly $1 \cdot A = A$ for each $A \in LC(V)$ (see Remark~\ref{rem:neutral}).   

For each $C$, $x,y$ by associativity of the  filter product we have  
\[ x (y     \hat C)= x \{ y\cdot z \colon \, z \in C\} = \{(x\cdot y)\cdot z \colon \, z \in C\} = (x\cdot y)  \hat C.\]
Also, by Claim~\ref{cl:BoJo},
\bc  $ x(y    \hat C) = x ( \hat{y\cdot C}) = \hat {x\cdot (y\cdot C)}$ and $(x\cdot y )   \hat C= \hat{(x \cdot y) \cdot C}$. \ec 
So $ \hat {x\cdot (y\cdot C)}=\hat{(x \cdot y) \cdot C}$. Then by Claim~\ref{subset is correct}(c) $x\cdot (y \cdot C) = (x \cdot y) \cdot C$ as required. 
\end{proof}
 Claim~\ref{axiom gammaV is a group action} is equivalent to the statement that $ (x,A) \to x \cdot A$ is an action of $\F(M)$ on $M$.  For, on the basis of  the   axioms so far, $M$ is partitioned into the sets $LC(V)$ for   $^*$-subgroups $V$,    the orbits of this action.    We have provided  the current  formulation of the claim mainly for notational convenience.
%
%

\begin{remark} \label{rem:identify with N} \ 

\n  {\rm Recall that we now regard an  abstract coset  structure $M$ as having   domain~$\omega$. So if $LC(V) \sub \omega$ is infinite we can identify its elements $A_0, A_1, \ldots$ with the natural numbers, and the action $\gamma_V$ can be viewed as an action on $\omega$.}  \end{remark}

\subsection{For  Roelcke precompact $\+ F(M)$,   each open coset    has the form~$\hat A$}

In this section we provide an important  tool. Introducing the  new  Axiom~\ref{axiom Roelcke}  for $M$, we show that if  the Polish group $\+ F(M)$ is Roelcke precompact (see the introduction), then  each open subgroup of $\+ F(M)$ is named by a $^*$subgroup in $M$.  This  will be   needed in Section~\ref{isomorphism is essentially countable} to  verify   that $\+ M( \+ G(M)) \cong M$ for each $M \in \+ B$,    where $\+ B$ is as in   Subsection~\ref{rem:plan} and $\+G(M)$ is a realization of $\+ F(M)$ as a permutation group. We will also apply this  tool to characterise the coarse groups of profinite groups  in Subsection~\ref{ss:profinite}. In this case, we could actually  use a simpler version of Axiom~\ref{axiom Roelcke} that only involves left cosets of a fixed subgroup, rather than double cosets.

Recall our letter conventions: letters $A$ to $F$ and their variants denote elements of $M$ (called $^*$cosets), and letters $U,V,W$ denote $^*$subgroups. 
Also recall from Definition~\ref{def:topo}  that  $\hat{A}=\{x\in \F(M)\colon A\in x\}$.    As always $M$ is a structure with domain $\omega$ in the language with one ternary relation symbol, and we generally  assume  that $M$ satisfies the (still growing) list of axioms.

%
%
%
%
%

Note that by Claim~\ref{subset is correct}  in Section~\ref{isomorphism is essentially countable}   that   the map $A \mapsto \hat A$ is a 1-1 map   from       elements of $M$ to open  cosets  of $\F(M)$. 
  After adding Axiom~\ref{axiom Roelcke}, we will show in Lemma~\ref{all cosets appear in M} that this map is onto,  assuming that $\F(M)$ is  Roelcke precompact:  each open coset  in  $\F(M)$ is of the form $\hat{A}$ for some $A \in M$. 
  
  We begin with  an auxiliary claim.
\begin{claim} \label{correct form of left cosets} 
For any left coset $x\hat{V}$ in $\F(M)$, there is a left $^*$coset $A$ of $V$ in $M$ such that  $x\hat{V}=\hat{A}$. 
\end{claim} 
\begin{proof} 
Since $x$ is a full filter, there is some left $^*$coset $A$ of $V$ in $x$. 
We claim that $x\hat{V}= \hat{A}$. We have $x\hat{V}\subseteq \hat{A}\hat{V}=\hat{V}$, since $A\in x$ and $\hat{A}$ is a left coset of $\hat{V}$ by Claim \ref{subset is correct}. To see that $\hat{A}\subseteq x\hat{V}$, let $y\in \hat{A}$. Since $x,y\in \hat{A}$, $x^{-1}y\in \hat{A}^{-1} \hat A=\hat{A^\diamond} \hat A\sqsubseteq \hat{V}$ by Claim \ref{subset is correct}. Thus $y  \in x\hat{V}$. 
\end{proof}

Consider any open subgroup $\U$ of $\F(M)$. Since $\U$ is open and $1_{\F(M)}\in \U$, there is an  $A$ in $M$ with $1_{\F(M)}\in \hat{A}$ and $\hat{A}\subseteq \U$. Now $A$ is equal to a $^*$subgroup $V$ in $M$, since $1_{\F(M)}$ contains only $^*$subgroups by Axiom \ref{axiom disjointness of cosets} and directedness of full filters. By Roelcke precompactness of $\F(M)$, the subgroup  $\U$ is a union of finitely may double cosets of the form $\hat{V}x\hat{V}$.   
Then, 
by the foregoing claim,  $\U=\bigcup_{i<n} \hat{V}\hat{A}_i$ for some left $^*$cosets $A_i$ of $V$ in $M$. 

To reach our goal, the main point is to  show that  each open subgroup in $\F(M)$ is of the form $\hat{U}$ for some $^*$subgroup $U$ in $M$. It now suffices to  introduce an axiom ensuring that a finite union of double cosets of $\hat{V}$ that is closed under products and inverses equals  $\hat{U}$ for some subgroup $U$ in~$M$.  First we need to establish  three claims; each one asserts that a  certain semantic condition in $\+ F(M)$  is  first-order  definable in $M$. 
\begin{claim}[Formula $\phi$]  \label{express nonempty intersection}  There is a first-order formula $\phi$ such that \bc $M \models \phi(A,B,C) $ iff 
$\hat{A}\hat{B}\cap \hat{C} = \emptyset$. \ec \end{claim} 
\begin{proof} 
$\phi(A,B,C)$ expresses that    \bc there are \emph{no}  $D\sqsubseteq A$ and $E\sqsubseteq B$ with $D E\sqsubseteq C$.  \ec

For the implication from left to right, by  contraposition suppose  that $  \hat{A}\hat{B} \cap \hat{C} \neq\emptyset $. Let $x\in\hat{A}$ and $y\in \hat{B}$ with $xy\in \hat{C}$. Since the group operation on $\F(M)$ is continuous by Prop.\  \ref{filter group is Polish} and $\hat{C}$ is open, there are basic open subsets $\hat{D}\subseteq \hat{A}$ and $\hat{E}\subseteq \hat{B}$ with $\hat{D}\hat{E} \subseteq \hat{C}$. 
Then $D E\sqsubseteq C$ by Claim \ref{subset is correct}. 

For the    implication from right to left, by contraposition suppose that  $D\sqsubseteq A$,  $E\sqsubseteq B$ and $DE \sqsubseteq C$. Then every $x\in \hat{D} \hat{E}$ is an element of $\hat{A}\hat{B}\cap\hat{C}$. 
\end{proof} 

In the following, we repeatedly use that double cosets of $\F(M)$ of the form $\hat V \hat A$, where $A$ is a left $^*$coset of $V$,   are clopen. This follows from  the hypothesis that $\F(M)$ is Roelcke precompact.
\begin{claim}[Formulas $\psi_n$] 
For each $n\geq 1$, there is a first-order formula $\psi_n$ such that  \bc $M \models  \psi_n(A_0,\dots,A_{n-1},B,V)$  iff  $\hat{B}\subseteq \bigcup_{i<n}\hat{V}\hat{A}_i$,  \ec for all $B$ and all  left $^*$cosets $A_0,\dots,A_{n-1}$ of $V$. 
\end{claim} 
\begin{proof} 
$\psi_n(A_0,\dots,A_{n-1},B,V)$ expresses that  \bc there is \emph{no}  $C\sqsubseteq B$ such that for all $i<n$, $  \phi(V,A_i,C)$.  \ec

First suppose that  $\hat{B}\not\subseteq \bigcup_{i<n}\hat{V}\hat{A}_i$. Since $\bigcup_{i<n}\hat{V}\hat{A}_i$ is clopen, there is some $C$ with $\hat{C}\subseteq \hat{B}$ and $\bigwedge_{i<n}\hat{V}\hat{A}_i \cap \hat{C}=\emptyset$. Then $C\sqsubseteq B$ by Claim \ref{subset is correct},  and $\bigwedge_{i<n}\phi(V,A_i,C)$ by Claim \ref{express nonempty intersection}. 

Conversely, assume that there is some $C\sqsubseteq B$ with $\bigwedge_{i<n} \phi(V,A_i,C)$. Then $\hat{C}\sqsubseteq \hat{B}$ by Claim \ref{subset is correct} and $\bigwedge_{i<n}\hat{V}\hat{A}_i\cap \hat{C}=\emptyset$ by Claim \ref{express nonempty intersection}. Hence $\hat{B}\not\subseteq \bigcup_{i<n}\hat{V}\hat{A}_i$. 
\end{proof} 

\begin{claim}[Formulas $\theta_n$]  
For each $n\geq 1$, there is a first-order formula $\theta$ such that  \bc    $M \models \theta_n(A_0,\dots,A_{n-1},V)$ iff  $\bigcup_{i<n}\hat{V}\hat{A}_i$ is a subgroup of $\F(M)$, \ec  for all left $^*$cosets $A_0,\dots,A_{n-1}$ of $V$. 
\end{claim} 
\begin{proof} 
Note that   $\hat{V}\hat{A}_j \hat{V}\hat{A}_l= \hat{V}\hat{A}_j \hat{A}_l$ for all $j,l<n$. 

We first express that    $\bigcup_{i<n} \hat{V}\hat{A}_i$ is closed under products. We will show that for all $j,l<n$, the statement $\hat{V}\hat{A}_j\hat{A}_l\not\subseteq \bigcup_{i<n}\hat{V}\hat{A}_i$ is equivalent to the following first-order formula $\rho_n(A_0,\dots,A_{n-1},V)$ in $M$: 
there are $B\sqsubseteq V$, $C\sqsubseteq A_j$, $D\sqsubseteq A_l$ and $E,F$ with $B C\sqsubseteq E$, $E D\sqsubseteq F$ and $\bigwedge_{i<n} \phi(V,A_i,F)$. 

Suppose first that $\rho(A_0,\dots,A_{n-1},V)$ holds in $M$ via $D, E$ and $F$. Take any $x\in \hat{B}$, $y\in \hat{C}$ and $z\in \hat{D}$. Then $x\cdot y\cdot z\in \hat{V}\hat{A}_j\hat{A}_l$ by Claim \ref{subset is correct} and by hypothesis $x\cdot y\cdot z\in\hat{F}$. Since $M \models \bigwedge_{i<n} \phi(V,A_i,T)$, we have $x\cdot y\cdot z\notin \bigcup_{i<n}\hat{V}\hat{A}_i$ by Claim \ref{express nonempty intersection}. 

Suppose conversely that $\hat{V}\hat{A}_j\hat{A}_l\not\subseteq \bigcup_{i<n}\hat{V}\hat{A}_i$ and take some $x\in\hat{V}$, $y\in \hat{A}_j$ and $z\in\hat{A}_l$ with $x\cdot y\cdot z \notin \bigcup_{i<n}\hat{V}\hat{A}_i$. 
Since $\bigcup_{i<n}\hat{V}\hat{A}_i$   is clopen, there is $F$ disjoint from $\bigcup_{i<n}\hat{V}\hat{A}_i$ with $x\cdot y\cdot z\in \hat{F}$. 
By continuity in Prop.\  \ref{filter group is Polish}, there is $\hat{E}\subseteq \hat{V}\hat{A}_j$ such that  $x\cdot y\in \hat{E}$,  and $\hat{D}\subseteq\hat{A}_l$ such that  $z\in \hat{D}$ and $\hat{E}\hat{D}\subseteq \hat{F}$. 
Again by continuity, there is  $B\subseteq\hat{V}$ such that   $x\in \hat{B}$,  and $C\subseteq\hat{A}_j$ such that   $y\in \hat{D}$ and $\hat{B} \hat{D}\subseteq E$. Now $\rho_n(A_0,\dots,A_{n-1},V)$ holds via $D,E$ and $F$ by Claims \ref{subset is correct} and \ref{express nonempty intersection}. 

Similarly, one can express that  $\bigcup_{i<n} \hat{V}\hat{A}_i$ is closed under inverses using the~$\diamond$ operation in Axiom~\ref{axiom inverses 1}. We leave this case  to the reader.
\end{proof} 

We are now ready  to express  the next axiom about an $L$-structure $M$. It~is the conjunction of an infinite set of first-order sentences. Note that its  conclusion   is equivalent to $\bigcup_{i<n} \hat{V} \hat{A}_i=\hat{U}$. 
\begin{axiom} \label{axiom Roelcke} 
Let  $n\ge 1$. Let $A_i\in LC(V)$ 
for all $i<n$. 

\n If  $\theta_n(A_0,\dots,A_{n-1},V)$ holds, then there is a $^*$subgroup $U$ such  that  
 \bc $\bigwedge_{i<n} [ V A_i\sqsubseteq U]$ and $\psi_n(A_0,\dots,A_{n-1},U,V)$.  \ec
\end{axiom} 

\begin{lemma} \mbox{} \label{all cosets appear in M}  Suppose that the Polish group $\+F(M)$ is Roelcke precompact.
\begin{enumerate-(a)} 
\item 
Every open subgroup $\U$ of $\F(M)$  equals $\hat{U}$ for some $^*$subgroup $U$ in $M$. 
\item 
Every open coset   in $\F(M)$ equals $\hat{A}$ for some $A$ in $M$. 
\end{enumerate-(a)} 
\end{lemma} 
\begin{proof} 
\

\n  (a) As remarked above, there are a $^*$subgroup $V$ and left $^*$cosets $A_0,\dots, A_{n-1}$ of   $V$ such that $\U=\bigcup_{i<n}\hat{V}\hat{A}_i$.  Axiom \ref{axiom Roelcke}      yields a $^*$subgroup $U$ in $M$
such  that $\U=\hat{U}$. 
 
\n (b)  now follows from Claim \ref{correct form of left cosets}. 
\end{proof}

\subsection{The profinite case} \label{ss:profinite} Recall  from the introduction that all our topological groups are separable, and    that a topological  group is profinite if and only if it is isomorphic to a compact subgroup of $\S$.
We show that the coarse groups of profinite   groups can be characterised with only one further axiom, stated as part of Proposition~\ref{prop:profinite}. By the form this axiom takes, and  given that the foregoing axioms   either determine arithmetical classes or can be replaced by   axioms which do so, this shows that the class of such coarse groups (with domain~$\omega$)  is arithmetical. 

In the following recall that $U,V$ range over $^*$subgroups.  In Remark~\ref{rem:normal}  we discussed normal $^*$subgroups $V$, defined by the condition $LC(V)= RC(V)$. 
\begin{prop}  \label{prop:profinite} Suppose $M$ satisfies the Axioms through to \ref{axiom Roelcke} introduced so far.  Then $\+ F(M)$ is compact $\LR$  $M$ satisfies the condition

\hfill $\forall U \exists V \sqsubseteq U \, [ LC(V)= RC(V)] \ \land \
\forall U  \, [LC(U) \text{ is finite}]$. \end{prop}
\begin{proof}  $\RA$: Recall that  $\+ F(M)$ is totally disconnected. So, if    $\+ F(M)$ is compact, then  each open subgroup of $\+ F(M)$ contains a normal open subgroup. Since 
$\+ F(M)$ is Roelcke precompact, Claim~\ref{subset is correct} and Lemma~\ref{all cosets appear in M} imply the corresponding statement for~$M$.

\n  $\LA$:  By a construction similar to the one in Claim~\ref{build full filter},  let $\seq{N_k}\sN  k$ be a descending chain of normal $^*$subgroups such that $\forall U \exists k \, [ N_k \sqsubseteq U]$. Let $G_k$ be the group  induced by $M$ on $LC(N_k)$  as in Remark~\ref{rem:normal}. We define an onto map $p_k \colon G_{k+1} \to G_k$ as follows: given $A\in LC(N_{k+1})$,   using Axiom~\ref{axiom disjointness of cosets} let  $p_k(A)= B$ where $A \sqsubseteq B\in LC(N_k)$. Each $p_k$ is a homomorphism by Axioms~\ref{ax:order} and~\ref{axiom inverses 1}.

Let $G$ be the inverse limit:  $ G=\projlim_k (G_k, p_k)$. Thus \bc $G = ( \{f \in \prod_k G_k \colon \forall k \, f(k)= p_k(f(k+1))\}, \cdot)$,  \ec which is closed and hence  compact group subgroup of the Cartesian product of the $G_k$. We claim that $G \cong (\+F(M), \cdot)$ via the map $\Phi$ that sends $f \in G$ to the filter in $\+ F(M)$ generated by the $^*$cosets $f(k)$, namely \bc $\Phi(f) = \{C\in M \colon \, \exists k \, f(k) \sqsubseteq C\}$.      \ec  It is clear that $\Phi$ is a monomorphism. For continuity of $\Phi$ at~$1$, let $\+ U$ be an open subgroup of $\+ F(M)$. By Lemma~\ref{all cosets appear in M} there is a $^*$subgroup $U\in M$ such that $\hat U  =   \+ U$. Choose $k$ such that $N_k \sqsubseteq U$. We can view $U$ as a subgroup of $G_k$, and so  $\Phi^{-1}(\+ U)= \{f \in G \colon \, f(k) \in U\}$ is open in $G$. 

To show $\Phi$ is onto, given a full filter $x\in \+ F(M)$, for each $k$ there is $f(k)= B_k \in LC(N_k)$ such that $B_k \in x$. Then $f \in G$, and clearly $\Phi(f)= x$. 

This shows that $\+ F(M)$ is compact as a continuous image of the  compact space $G$. 
\end{proof}
  We will return to the topic of compact subgroups of $\S$  in Subsection~\ref{ss:SBB compact}. 
%

\section{Isomorphism of oligomorphic groups,  and countable models}

   \label{isomorphism is essentially countable} 

\n The main result of this section establishes that an oligomorphic group $G$ can   in a Borel way be interchanged with    a structure with domain $\omega$, namely its corresponding coarse group $\+ M(G)$. 

\begin{theorem} \label{thm:main} Isomorphism of oligomorphic subgroups of $S_\infty$  is classwise Borel bireducible with 
the   isomorphism  relation  on an invariant Borel set of countable  structures in a finite signature.  \end{theorem}


 \subsection{Review  of the  result of Kechris, Nies and Tent} \label{ss:KNT review} 


 Before proving the theorem, we    need to  review  in some more detail the map $\+ M$ defined in Kechris et al.\ \cite[Section 3.3]{Kechris.Nies.etal:18}. This map   shows that isomorphism of    Roelcke precompact groups   is Borel  {reducible} to  isomorphism on the set of   $L$-structures with domain $\omega$,  for the language $L$ with one ternary relation symbol $R$.   These  structures form a Polish space  $X_L = \+ P( \omega \times \omega \times \omega)$, the sets of triples of natural numbers. 
%
%
\cite{Kechris.Nies.etal:18}  provides a Borel map $\+ M$ from the set of  Roelcke precompact closed subgroups of $\S$ to structures in $X_L$. 
For such a group $G$, the set $\+ N_G$  of all open subgroups of $G$  is countable; 
we think of the domain of the structure $\+ M(G)$ as consisting  of the  cosets of subgroups in $\+ N_G$  (this structure is denoted by~$M_G$ in~\cite{Kechris.Nies.etal:18}).  
Then, by a result of Lusin-Novikov in the version of  \cite[18.10]{Kechris:95}, one can in a Borel way find a bijection between these cosets and $\omega$.  
  
  We note that this approach also works for Borel classes of groups  where $\+ N_G$ is merely a  countable neighbourhood basis of $1$ consisting of open subgroups such that $\+ N_G$ is isomorphism invariant; for instance, $\+ N_G$ could consist of the the compact open subgroups in a locally compact subgroup $G$ of $\S$.

\subsection{Plan of the proof}  \label{ss:plan} We will introduce a Borel inverse  up to isomorphism of  the map $\+ M$, restricted to oligomorphic groups. In more detail, let   $\+ B$ be the closure under isomorphism 
of the range of $\+ M$ on the class of oligomorphic groups. 
We will show that $\+ B$ is Borel, and  define a Borel map $\+ G$ from  $\+ B$ to the class of oligomorphic closed subgroups of $\S$  such that for each oligomorphic closed subgroup $G$ of $\S$,  and each structure $M$ in $\+ B$, we have 
\begin{equation} \label{eqn:invert}  \+ G(\+ M(G) )  \cong G \text{ and } \+ M( \+ G(M)) \cong M. \end{equation}
%

We will have $M\cong N \Leftrightarrow \+ G(M) \cong \+ G(N)$ for all $M,N\in \+ B$, as required for the proof of Theorem~\ref{thm:main}.
The implication $M\cong N\Rightarrow \+ G(M) \cong \+ G(N)$ will follow from the definition of the map $\+ G$. The reverse implication  will follow from (\ref{eqn:KNS}) and (\ref{eqn:invert}). 


The group $\+ G(M)$ is obtained from $\+ F(M)$ by specifying in a Borel way an embedding as an oligomorphic  closed subgroup of $\S$. To carry this out, 
we will add further    ``axioms"  that hold for   all the structures of the form $\+ M(G)$, where $G$ is oligomorphic. As before they  can be expressed by  monadic $\Pi^1_1$ sentences     or $\omc$ sentences in  the signature with one ternary relation symbol. A class $\+ C$ of $L$-structures will be    defined as the set of structures satisfying all these axioms. Then $\+ C$ is  $\Pi^1_1$. For an $L$-structure $M$ in $\+ C$  we will be able to recover an oligomorphic  group $\+ G(M) $ via a Borel map   in such a way that (\ref{eqn:invert}) holds. This implies that $\+ C$ equals $\+ B$, the closure of $\ran (\+ M)$ under isomorphism (which is analytic), so $\+ B$ is Borel.

\subsection{Ensuring that $\F(M)$ is isomorphic to an oligomorphic closed subgroup of $S_\infty$}

Given  a Polish group $G$ with a faithful  action $\gamma \colon G \times \omega \to \omega$, we obtain a monomorphism $\Theta_\gamma \colon G \to \S$ given by $\Theta_\gamma(g)(k)= \gamma(g,k)$. A Polish group action is continuous if and only if it is separately continuous. In the case of an action on $\omega$ (with the discrete topology), the latter condition  means that for each $k,n \in \omega$, the set $\{ g \colon \gamma(g,k)=n\}$ is open. So $\gamma$ is continuous if and only if $\Theta_\gamma$ is continuous. 

\begin{definition} 
We say that a faithful  action $\gamma \colon G \times \omega \to \omega$   is \emph{strongly continuous} if the embedding $\Theta_\gamma$ is topological. 
\end{definition} 

Equivalently, the action is continuous and for each neighbourhood $U$ of $1_G$,  the set  $\Theta_\gamma(U)$ is open in $\Theta_\gamma(G)$, namely, there is $n$ such that $\fa k < n \, \gamma(g,k) = k$ implies $g \in U$. Strong continuity implies that $G$  is topologically isomorphic to a closed subgroup of $S_\infty$. Clearly, not every continuous action is strongly continuous;  for instance let $G$ be  the discrete group of permutations of finite support and take  the natural  action of $G$ on $\omega$. 

We will   introduce axioms that ensure that $\F(M)$ has an action on $\omega$ that~is \bc   (a)~faithful,  (b)~oligomorphic, and (c)~strongly continuous. \ec
By the following lemma, each  oligomorphic group $G$ has an open subgroup $W$ so that the natural action of $G$ on the set $LC( W)=G/W$  of   left cosets of $W$ has these three properties. 
In the general setting of a coarse group structure $M$ we ensure
  the existence of a subgroup with these properties by a further axiom.

\begin{lemma} \label{canonical oligomorphic action} 
{Let $G$ be an  oligomorphic  closed subgroup  of $S_\infty$. There is an open subgroup $W$ such that the left translation action $\gamma \colon G\curvearrowright LC( W)$   is faithful and oligomorphic.  Furthermore, for any    listing  without repetition $\seq {A_i}\sN i$  of the cosets of $W$,  when viewing $\gamma $ as an action on $\omega$ via this listing, this action is strongly continuous. }
\end{lemma} 
\begin{proof} 
    Let $x_1, \ldots, x_k \in \omega$ represent the 1-orbits of $G$. Let $W$ be the pointwise stabiliser of $\{x_1, \ldots, x_k\}$. If $g \in G - \{1\}$ then there are $p \in G$ and $i \le k$ such that $g \cdot (p \cdot x_i) \neq p\cdot x_i$. So $p^{-1} g p \not \in W$, and hence $g \cdot  p W \neq pW$.   In particular,  the action is faithful, and hence  $LC( W)$ is infinite.

 Choose $a_i \in \S$ such  that $A_i = a_i W$. To show  that $\Theta_\gamma$  is continuous, given $n$, let $U = \bigcap_{i< n} a_i W a_i^{-1}$, and note that $U $ is an   open subgroup  of $G$. Then $g \in U$ implies $\Theta_\gamma(g)(i) = i$ for $i < n$. 
 
 To show  that $\Theta_\gamma^{-1}$  is continuous, given $n$, for each $i< n$ choose ${p(i)}\in \omega$   such that $i =  a_{p(i)} x_r$ for some $r$. If $\Theta_\gamma(g)$ fixes all the numbers ${p(i)}$ then $\gamma(g , i) = i$ for each $i<n$. 
 
 Since  $G$  is oligomorphic, it is Roelcke precompact. Then, since the  action of $G$  on $LC(W)$ is strongly  continuous and has finitely many 1-orbits,  by Tsankov \cite[Thm 2.4]{Tsankov:12} this  action is oligomorphic.  
\end{proof} 

\begin{remark} \label{rem: abc} Given a $^*$subgroup $V$, we discuss how to express that $\gamma_V$ has properties (a), (b) and  (c) above      via either $\Pi^1_1$ formulas  or $\omc$ formulas, in the signature~$L$.

\medskip

\n (a) We can say   that $\gamma_V$ is faithful by expressing the following  by a   $\Pi^1_1$ formula: for all $x\neq 1$, there are   disjoint left $^*$cosets $A$, $B$ of $V$ such that  $x\cdot A=B$.  Note that this makes $LC(V)$ infinite.

\medskip

\n (b) To say that  $\gamma_V$ is  oligomorphic using a formula in $\omc$, we can require  that    for all $k\geq1$, there is some $n\geq1$ and there are $k$-tuples $\vec{C}^0,\dots,\vec{C}^{n-1}$ of left $^*$cosets of $V$ with the following property. For each $k$-tuple $\vec{B}$ of left $^*$cosets of $V$, there is some $i<n$ and some $S$ such that for all $j<k$, we have $S B_j\sqsubseteq C_j^i$. To show that this condition      implies that $\gamma_V$ is oligomorphic,  choose any $x$ such that  $S\in x$. Then $x \cdot B_j = C^i_j$ for each~$j$.

  If $M$ satisfies the  given   condition, we say for short  that $M$ is \emph{formally oligomorphic}.

\medskip
\n (c) We can express that $\gamma_V$ is strongly continuous by an $\omc$ formula.   Write $\Theta_V$ for $\Theta_{\gamma_V}$. First note that $\Theta_V$ is automatically continuous at $1$ (and hence continuous): a basic neighbourhood of $1$ in $\S$ has the form $\{ \rho \colon \, \fa i< n \,[  \rho(i) = i ]\}$.  
For a full filter $x$, we have $xA_i = A_i$ if and only if $x  \in \hat S$ for some $S$ such that $SA_i\sqsubseteq A_i$. So by the definition of the topology~\ref{def:topo}, $\bigcap_{i< n} \{ x \colon \, xA_i = A_i\}$ is an open subset in $\F(M)$ that is mapped by $\Theta$ into that neighbourhood. \end{remark}

That $\Theta_V^{-1}$ is continuous at $1$ means  the following:
 \begin{equation} \label{eqn: fifi} \forall U  \ex k \ex  B_1, \ldots , B_k  \in LC(V)  \fa x \,  [ \bigwedge_i x \cdot B_i = B_i \to U \in x].\end{equation}
To avoid the universal second-order quantifier $\forall x$, we will instead use
\begin{equation} \label{eqn: fifi2}\forall U  \ex k \ex  B_1, \ldots , B_k  \in LC(V)  \fa S \,  [ \bigwedge_i [S   B_i \sqsubseteq  B_i] \to S \sqsubseteq  U ]. \end{equation}  
 
\begin{claim} Let $k \in \NN$. Given $V , U, B_1, \ldots, B_k  \in M$,  we have 
\bc $ \fa x \,  [ \bigwedge_i x \cdot B_i = B_i \to U \in x] \LR \fa S \,  [ \bigwedge_i [S   B_i \sqsubseteq  B_i] \to S \sqsubseteq  U]$. \ec
Thus for each $V$,  (\ref{eqn: fifi})   is equivalent to  (\ref{eqn: fifi2}). \end{claim}
\begin{proof} We make use of Claim~\ref{subset is correct}(c). For the implication~``$\RA$" suppose that $ \bigwedge_i [S   B_i \sqsubseteq  B_i] $. Let $x$ be a full filter such that $S\in x$. Then $x\cdot B_i=B_i$ for each $i\le k$, and hence $U\in x$. So $ \hat S \sub \hat U$ and hence $S \sqsubseteq U$.

\n For the implication~``$\LA$" suppose that $ \bigwedge_i [x \cdot   B_i =  B_i] $. The by downward directness of full filters, there is $S \in x$ such that $\bigwedge_i [S   B_i \sqsubseteq  B_i]$. So $S \sqsubseteq U$, whence $U\in x$ by upward closure of full filters.  \end{proof}

\begin{axiom} \label{axiom faithful oligomorphic action} 
There is a $^*$subgroup $W$ in $M$  such that $\gamma_W$ is faithful,   formally oligomorphic, and strongly  continuous. 
\end{axiom} 
Later on in Section~\ref{s: turning}, we will argue that we can determine such a $W$ via a Borel function applied to $M$. Then we will define the required oligomorphic group $\+ G(M)\cong \+ F(M) $ 
as the range of $\Theta_W$. The first statement in  (\ref{eqn:invert}) will then follow from Prop.~\ref{fact:standard}.  In Section~\ref{ss:EC} we will reformulate the axiom in order to avoid the universal second order quantifier we are using to express faithfulness.

 Lemma~\ref{canonical oligomorphic action} together with the following claim ensures that $\+ M(G)$ satisfies Axiom~\ref{axiom faithful oligomorphic action}.
\begin{claim} \label{characterization of oligomorphic} 
If $M=\+ M(G)$ and $V$ is a $^*$subgroup in $M$ such that $\gamma_V$ is oligomorphic, then $\gamma_V$ is formally oligomorphic. 
\end{claim} 
\begin{proof} 
Since $\gamma_V$ is oligomorphic, we have some $x\in\F(M)$ such that for all $j<k$, $x\cdot B_j= C_j^i$ in the notation above. It is easy to see that the action $\F(M)\curvearrowright \F(M)/V$ induced by the group operation on $\F(M)$ satisfies $x\cdot \hat{B}_j= \hat{C}_j^i$. 
By continuity of the group operation in Prop.\  \ref{filter group is Polish}, there is some $S$ such that $\hat{S}$ contains $x$ and for all $j<k$, we have 
$\hat{S} \hat{B}_j \subseteq \hat{C}^i_j$ and hence $S B_j \sqsubseteq C^i_j$ by Claim \ref{subset is correct}. 
\end{proof}


\subsection{Turning the filter group into a closed subgroup of $\S$} \label{s: turning}
We  now define the Borel map $\+ G$.  Let $\+ C$ be the set of $L$-structures $M$ with domain $\omega$ that satisfy the axioms stated above. Note that $\+ C$ is $\Pi^1_1$ because all axioms can be expressed in $\Pi^1_1$ form or in $L_{\omega_1, \omega}$ form. Also, $\+ C$ contains the closure under isomorphism of the range of the map~$\+ M$, denoted $\+ B$ in Section~\ref{ss:plan} above.


As mentioned above, the   relation
 $\{ \la M, W \ra \colon M \in \+ C \lland    W \in M$   is a $^*$subgroup in $M$ satisfying the properties  in Axiom~\ref{axiom faithful oligomorphic action}$\}$  is  $\Pi^1_1$.  By   $\Pi^1_1$-uniformization (Addison/Kondo, see e.g.\ \cite[Theorem 4E.4]{Moschovakis:80})  there is a   function $f\colon \+ C \rightarrow \omega$ with  $\Pi^1_1$ graph that sends each $M \in \+ C$ to some $W\in M$  of this kind.  
 Recall that the embedding $\Theta_V$,  for certain $^*$subgroups $V$ in~$M$,  was  defined in (c) before Axiom~\ref{axiom faithful oligomorphic action}.  We define $\+ G(M)$ as the range of~$\Theta_W$ where $W = f(M)$. In other words, $\+ G(M)$ is the closed subgroup of $S_\infty$ determined by the action of $\+ F(M)$ on $LC(W)$. Here we use the canonical increasing   bijection between $\omega$ and $LC(W)$ (an infinite subset of $\omega$)    to view the action on $LC(W)$ as  an action on~$\omega$, as specified in Remark~\ref{rem:identify with N}.

We are now ready to  establish (\ref{eqn:invert}), restated here for convenience: 

 \begin{prop}   \label{cl:verify 2} For each oligomorphic group $G$ and each structure \\ $M \in \+ C$, we have 
\begin{equation*}    \+ G(\+ M(G) )  \cong G \text{ and } \+ M( \+ G(M)) \cong M. \end{equation*} \end{prop} 
\begin{proof} 
As already mentioned, the first statement   follows from Prop.~\ref{fact:standard}.
Given $A \in M$, we view $\hat A$ now as an open coset of $\+ G(M)$,  rather than of the filter group $\+ F(M)$. 

By Axiom \ref{axiom faithful oligomorphic action}, there is a  $^*$subgroup $W$ in $M$ such that  $\gamma_W$ is faithful,  oligomorphic and yields a topological  embedding into $\S$. 
Since $\+ F(M)$ is a Polish group by Claim~\ref{filter group is Polish}, the  range of $\gamma_V$ is  an oligomorphic closed subgroup of  $\S$. 
  Hence $\F(M)$ is Roelcke precompact   by \cite[Theorem 2.4]{Tsankov:12}.
  
  Then, by Lemma \ref{all cosets appear in M},   the map $ A \mapsto \hat A$ is a bijection between $M$ and  $\+ M(\+ G(M))$. 
By Claim \ref{subset is correct} it is an isomorphism. 
Thus we  obtain   the second statement. 
\end{proof} 
Note that we actually show for each $A\in M$ that  $(\+ M( \+ G(M)), \hat A) \cong (M, A)$. This will be used below.

Proposition~\ref{cl:verify 2} implies that   $\+ B = \+ C$. Since $\+ B$ is the  closure    under isomorphism of the range of a Borel measurable map defined on  a Borel domain, it is analytic. Since  $\+ B$ is also coanalytic, it is   Borel. 
Since the domain of $f$ is Borel, the  graph of $f$ is analytic because  $f(x)\neq n$ iff $\exists m\neq n\ f(x) =  m$. So the graph of $f$ is Borel.

%
%

 Note that $\G(M)$ is an element of the Effros Borel space of $S_\infty$ (see Section~\ref{s:prelim}).    In the following, $\sss$ will denote an injective map on  initial segments of the integers, that is, 
 a tuple of integers without repetitions. 
 Let $[\sss]$  denote the  set of permutations extending~$\sss$:  \bc $\+ [\sss] = \{ f \in S_\infty \colon \sigma \prec f\}$  \ec (this is often denoted  $\+ N_\sss$ in the literature).  The sets $[\sss]$ form a basis for the topology of pointwise convergence of $S_\infty$. 

\begin{claim} The map $M \mapsto \+ \G(M)$,     for $M \in \+ C$,  is Borel. \end{claim} 
\begin{proof} 
Let  $[\sigma]$ be  an arbitrary  basic open subset of $S_\infty$. It is sufficient to show that $\{M\mid \G(M) \cap [\sigma]\neq\emptyset\}$ is Borel.  From $M$ we obtain $W= f(M)$ in a Borel way, and then the  list $A_0, A_1, \ldots $ of  left $^*$cosets of $W$  in ascending order.
We have \bc $\G(M)\cap [\sigma]\neq\emptyset \Longleftrightarrow \exists S\in M\ \fa i, j [ \sigma(i)=j  \to S A_i\sqsubseteq A_j]$,  \ec by the definition of the action $\gamma_W$  in~(\ref{eqn:gamma}).
\end{proof}

This completes the proof of Theorem~\ref{thm:main}.

 \subsection{Borel duality with coarse groups in the profinite case} \label{ss:SBB compact}  

Let $\+ C_\text{pro}$ be the class of coarse groups satisfying Axioms  through to \ref{axiom Roelcke} (but not necessarily~\ref{axiom faithful oligomorphic action}), as well as the   condition   in Proposition~\ref{prop:profinite} which is equivalent to compactness of $\F(M)$. Here we show that $\+ C_\text{pro}$ is Borel, and that the compact subgroups of $\S$ are classwise Borel bireducible to 
$\+ C_\text{pro}$, in analogy  with  the Borel version of Stone duality discussed above.

 Given $M \in \+ C_\text{pro}$,  let $\Theta \colon \+ F(M) \to \S$ be the map    induced by the   action of $\+ F(M)$ on $M$ given by Claim~\ref{axiom gammaV is a group action} and the discussion thereafter. This action  clearly is  faithful, so $\Theta$ is an embedding. As in Remark~\ref{rem: abc}(c), $\Theta$ is automatically continuous.  Then, since~$\+ F(M)$ is compact,   $\Theta$ is a topological embedding. Let $\+ G_\text{pro}(M)$ be its range, a closed subgroup of $\S$.  

It is easy to establish the analog of Proposition~\ref{cl:verify 2} using Lemma~\ref{all cosets appear in M}. This implies that  $\+ C_\text{pro}$ is Borel (alternatively this holds by the first-order reformulation of   Axiom~\ref{axiom product of filters} given in Subsection~\ref{son_quatsch}). We obtain:

\begin{prop} The Borel operators $\+ M$ and $\G_\text{pro}$ establish classwise Borel bireducibility 
between the compact subgroups of $\S$, and the coarse groups satisfying the Axioms  \ref{axiom basics} to \ref{axiom Roelcke} and the  condition   in Proposition~\ref{prop:profinite}. \end{prop}

\section{Complexity of the isomorphism relation between oligomorphic groups}


\subsection{Conjugacy} We begin with an easy result:
  conjugacy of oligomorphic  groups is smooth, that is, Borel reducible to the identity on $\mathbb R$. 
  
   For a closed subgroup $G$ of $ \S$, let $\+ E_G$ denote the orbit equivalence structure with domain $\omega$. For each $n$,  the signature of this  structure has a $2n$-ary relation symbol, denoting the orbit equivalence relation for the action of $G$ on ${}^n\omega$.  

 The following fact   holds in  general.

 \begin{fact} \label{fact:conj} Let $G$ and $H$ be   closed subgroups of $\S$.  Let $\alpha \in \S$. Then
 
 \bc $G, H$ are conjugate via  $\alpha$ $\LR$ $\+ E_G \cong \+ E_H$ via $\alpha$. \ec
 \end{fact}
 
 \begin{proof}
 \rapf  This is immediate.
 
 \lapf Let $M_G$ be the canonical structure for $G$, namely there are $k_n \le \omega$ many $n$-ary relation symbols, denoting the individual $n$-orbits. Let $M_H$ be the structure in the same signature where the  equivalence classes of $\+ E_H$  on ${}^n \omega$ are named so that $\alpha$ is an isomorphism $M_G \cong M_H$. Since  $G\le  \Aut(M_G)$, and $G$ is closed and dense, we have $G= \Aut(M_G)$; similarly,    $H = \Aut(M_H)$. Furthermore, $\alpha^{-1} \Aut(M_H) \alpha = \Aut(M_G)$.
 \end{proof}

 \begin{prop} The conjugacy relation between  oligomorphic groups is smooth. \end{prop}
\begin{proof}  The map $G \mapsto \+ E_G$  defined on closed subgroups of $\S$ is   Borel because one can in a Borel way find a countable dense subgroup of $G$, which of course  has the same orbits; based on  that subgroup one can directly construct $\+ E_G$.   

 For  countable structures $S$ in  a fixed countable  language,   mapping   $S$ to its theory $\text{Th}(S)$ is Borel.   The theory can be seen as  a subset of $\omega$, assuming a suitable encoding of the language.

 Suppose now that $G$ and $ H$ are oligomorphic closed subgroups of $\S$.  Note that $\+ E_G$ is interpretable without parameters in the   canonical structure~$N_G$ mentioned in Section~\ref{s:bi}. So  $N_G$ is $\omega$-categorical,  and hence $\+ E_G$ is  $\omega$-categorical as well.
 
 By  Fact~\ref{fact:conj} and since $\+ E_G$ and $\+ E_H$ are $\omega$-categorical,   
\bc  $G, H$ are conjugate   $\LR$ $\+ E_G \cong \+ E_H$   $\LR \Th ( \+ E_G) = \Th(\+ E_H)$,  \ec which  shows smoothness. 
\end{proof}


\subsection{Essential countability of the isomorphism relation} \label{section essential countabilility} 
    Recall that an equivalence relation  $E$ on a Polish space is called \emph{countable} if every equivalence class is countable. One says that $E$ is \emph{essentially countable}  if $E$ is Borel reducible to  a countable Borel equivalence relation. 
    
    We show that the isomorphism relation between   oligomorphic subgroups of $S_\infty$ is essentially countable.  As mentioned in the introduction, we apply   a result of Hjorth and Kechris \cite[Theorem 4.3]{Hjorth.Kechris:95} about  Borel invariant classes~$\+ C$ of countable structures. Given a finite signature, a subset $F$ of $\omc$ is called a \emph{fragment} if it is closed under syntactic  first-order operations such as quantification over elements, or substitution. Suppose first that we had a countable fragment $F$ such that each $M \in  \+ C$ is determined up to isomorphism among the countable structures  by $\Th_F(M)$, its theory in this fragment. Then $\cong_\+ C$ is smooth, because the map $M \mapsto \Th_F(M)$ is Borel.

  Their result uses a weaker hypothesis to yield a weaker conclusion. In  \cite[Theorem 4.3, (iii)$\to$(i)]{Hjorth.Kechris:95} they prove the following. 
   {Suppose  that there is a fixed  fragment $F$ as follows:   each $M \in \+ C$ contains   a tuple  of constants $\ol a$   such that $(M, \ol a)$  is determined up to isomorphism  among the countable structures  by $\Th_F(M, \ol a)$ (i.e, $\Th_F(M, \ol a)$ is $\aleph_0$-categorical).  Then  $\cong_\+ C$ is essentially countable.}
 
  Their proof proceeds as follows. They  need to obtain a countable Borel equivalence relation $E$ on a Borel space $Y$ so that $E$ is Borel  above $\cong_{\+ C}$.  The points of the Borel space $Y$ are $F$-theories of countable   models extended by finitely many constants.  Two theories are $E$-equivalent if they can be realised over isomorphic models in the language of $F$. (They verify as part of   their proof that $Y$ is indeed a Borel space on which $E$ is Borel.) The Borel reduction maps $M$ to $\Th_F(M,\ol a)$ where $\ol a$ is chosen so that   $\Th_F(M,\ol a)$ is $\aleph_0$-categorical. This is possible by a result in descriptive set theory due to Lusin-Novikov: one can in a Borel way uniformise a Borel relation that relates each $x$ to only countably many elements (see e.g.\  \cite[18.10]{Kechris:95}).

  Recall from  Lemma \ref{canonical oligomorphic action} that each   oligomorphic group  $G$ has an open subgroup $W$ such that the left translation action of  $G$ on the left cosets of $W$ is  oligomorphic, and yields a topological embedding of $G$ into $\S$.   The idea in  applying the Hjorth-Kechris result is now as follows. Given a structure~$M$  for the signature with one ternary relation satisfying the axioms so far,    require the existence of $W$ axiomatically for the action of the filter group on the (abstract)  left cosets of $W$. If $F$ is   the least fragment containing all the relevant formulas used in the axioms, then it can be shown that $(M,W) $,  for $W$ as above, is determined by its theory in $F$. Thus, $(W)$ is the tuple of constants one   adds to satisfy the hypothesis of the Hjorth-Kechris result. 
  \label{ss:EC}
 \begin{thm}  \label{thm:EC} The isomorphism relation between   oligomorphic aubgroups of $S_\infty$ is essentially countable. 
 \end{thm}
 
 \begin{proof}     Recall that $R$ is a ternary relation symbol.  Also recall that in Section~\ref{ss:plan} above  we denoted by $\+ B$  the closure under isomorphism of the range of the map~$\+ M$.  We showed in Section~\ref{s: turning} that $\+ B$ is Borel. So by the L\'opez-Escobar theorem there is   $\sss \in L_{\omega_1, \omega}(R)$ such that $M \in \+ B \LR M \models \sss$ for each model $M$. Let $\cong_\sigma$ denote the isomorphism relation on $\+ B$.  
 
 Let $F$  be the smallest  fragment of $L_{\omega_1, \omega}(R)$ containing $\sss$. Note that $F$ is countable. For a structure $M$   and $n$-tuple $\ol a$ in $M$, by  $\Th_F(M, \ol a)$ one denotes $\{ \phi(x_1, \ldots x_n) \in F \colon  (M, \ol a) \models \phi\}$. 
 
  By  Hjorth and Kechris \cite[Theorem 4.3]{Hjorth.Kechris:95}, the following are equivalent.
  
  \bi \item[(i)] $\cong_\sigma$ is essentially countable  \item[(ii)] for each $M\in \+ B$ there is  a tuple $\ol a $ in $M$ such that $\Th_F(M, \ol a)$ is $\aleph_0$-categorical. \ei
  We will verify (ii), where the tuple $\ol a$  has length 1: it  consists of the witness~$W$ for a stronger    version of  Axiom~\ref{axiom faithful oligomorphic action}.  The problem with our formulation of faithfulness in that axiom is that it is only $\Pi^1_1$ and hence cannot be used in a fragment. Instead, let $\delta(V)$ denote the following first-order formula, which implies  that $\gamma_V$ is faithful, as will be verified shortly: 

\medskip 
$\fa U\     \fa A \in LC(U) \setminus \{U\}\ \ex U' \sq  U $ 

\ \ \ \ \ \ \ $\fa A' \sq A, A' \in LC(U')\ \ex C \in LC(V)\ \ex D\in LC(V)\setminus\{C\}\  A'C\sqsubseteq D$. 

\begin{axiom}[Replaces Axiom~\ref{axiom faithful oligomorphic action}] \label{axiom faithful Borel version} 
There is a $^*$subgroup $W$ in $M$  such that   $M \models \delta(W)$, and  the action  $\gamma_W$ defined in (\ref{eqn:gamma})  is formally oligomorphic and strongly  continuous.
 \end{axiom} 

We claim that this  condition holds in $\+ M(G)$, for any    oligomorphic closed subgroup $G$  of $S_\infty$. 
By Lemma~\ref{canonical oligomorphic action} we may assume that the action of $G$ on~$\omega$  has a single 1-orbit.
Let $W=G_{0}$, the stabilizer of $0$. 

Suppose we are given an open subgroup $U$ of $G$, and let  $A \in LC(U)$, $A \neq U$.  By definition of the subspace topology on $G$,  there is tuple $\vec{y}$ of natural numbers such that  $U'=G_{\vec{y}}$ is contained in $U$. Take a left coset $A'=g U'\sqsubseteq A$. Since $A \neq U$ we have  $A' \neq U'$, and hence $g(y_j) \neq y_j$ for some $j$, say $j=0$. 
 Let  $h\in G$ with $h(0)=y_0$.    By definition of $g $ and since $U' \le h W h^{-1}= G_{y_0}$, 
 
 \bc $A'hW h^{-1} = g U' h W h^{-1}  = g h W h^{-1} \neq h W h^{-1} $. \ec
 Thus,  where  $C = hW$,   $A'C$  is a coset of $W$ different from $C$, as 
  required. 

\begin{claim}  \label{cl:ff}
If $\delta(W)$ holds, then $\gamma_W$ is faithful. 
\end{claim} 
\begin{proof} 
Suppose that $x\neq 1$ is a full filter of  $M$. Then there is a $^*$subgroup $U $ and  $A \in LC(U)$ such that  $A \in x$ and $A \neq U$.   We choose $U'$  as in the statement $\delta(W)$. Let  $A' $  be   the unique  $^*$coset  in $LC(U')$ such that $A' \in x$. Then $A ' \sq A$ by Axiom~\ref{axiom disjointness of cosets} and since $x$ is a filter. Choose $C\in LC(W) $ for this $A'$. Then $A'C \sqsubseteq D \neq C$, so  $x \cdot C \neq C$ as required.
\end{proof} 
Let $F$ be a countable  fragment of $L_{\omega_1, \omega}$ containing $\sss$, $\delta$ and the other  formulas needed to express Axiom~\ref{axiom faithful Borel version}.  The following now verifies Condition~(ii) in the Hjorth-Kechris theorem for this fragment. 
  \begin{claim} Suppose that $M, N \in \+ B$. Let  $W \in M$ be a  witness to Axiom~\ref{axiom faithful Borel version} for $M$. Let    $Z \in N$ be a $^*$subgroup  such that  $\Th_F(M,W ) = \Th_F (N,Z)$. Then $(M,W ) \cong (N,Z)$.   \end{claim}
  
  \begin{proof}	 Note that   $Z \in N$ is a  witness for  Axiom~\ref{axiom faithful Borel version} in  $N$ by definition of the fragment $F$.  Let $LC_M(W)$ denote the set of left $^*$cosets of $W$ in $M$, and similarly let $LC_N(Z)$ denote the set of left $^*$cosets of $Z$ in $N$; both sets are  identified with a set of natural numbers as explained in Remark~\ref{rem:identify with N}.  As in Condition (c) before Axiom~\ref{axiom faithful oligomorphic action} above, by    $  \+ G_W(M)$ we denote  the range of the natural embedding $\+ F(M) \to \S$ given by the action $\gamma_W$ of $\+ F(M) $ on $LC_M(W)$, and similarly for  $  \+ G_Z(N)$.  By the proof of Prop.~\ref{cl:verify 2}, we have  
  \begin{equation} \label{eqn:jjkj} (\+ M( \+ G_W(M)) , \hat W) \cong (M, W) \text{ and }(\+ M( \+ G_Z(N)) , \hat Z) \cong (N, Z). \end{equation} 
  Thus it suffices to show that the structures on the left sides are isomorphic. 
  
  Write $G=  \+ G_W(M)$ and $H =  \+ G_Z(N)$.
   As in   Fact~\ref{fact:conj}, let $\+ E_G$ and $\+ E_H$ be the corresponding orbit equivalence structures for the actions of $G$ on $LC_M(W)$ and of $H$ on $LC_N(Z)$. By our hypothesis we have $(M,W) \equiv  (N,Z)$ (i.e., the two structures have the same first-order theory).  By the definition of the group actions $\gamma_W$ and $\gamma_Z$, the structure $(\+ E_G, W)$ is interpretable in $(M,W)$, and similarly  $(\+ E_H, Z)$ is interpretable in $(N,Z)$ using the same collection of formulas.  This implies that  $(\+ E_G, W) \equiv (\+ E_H, Z)$. 
   
   Since $G$ is  oligomorphic,  the orbit equivalence  structures $\+ E_G$ and $\+ E_H$    are $\aleph_0$-categorical. Hence so are $(\+ E_G, W) $ and $ (\+ E_H, Z)$;  let $\alpha \in \S $ witness  that  $(\+ E_G, W) \cong (\+ E_H, Z)$.  
  
  As in   the proof of Fact~\ref{fact:conj},  $\alpha^{-1} H \alpha = G$. Since $\alpha(W)= Z$ and $\hat W$ is the stabiliser of $W$ and $\hat Z$ is the stabiliser of $Z$, we have $\alpha^{-1} \hat Z  \alpha = \hat W$. 
 Thus the map $B \mapsto \alpha B \alpha^{-1}$, for $B$ an open coset of $G$, is an isomorphism for the left hand side structures in (\ref{eqn:jjkj}),  as required. 
  \end{proof}

%
%
%

%
 
This completes the proof of Theorem~\ref{thm:EC}. 
  \end{proof}
  
%

\subsection{Extension of the upper bound  to the class of  quasi-oligomorphic groups}

  A closed subgroup $G $ of $ \S$ will be called \emph{quasi-oligomorphic} if it is (topologically) isomorphic to a an  oligomorphic subgroup $H$ of $  \S$. 
Note that $H$, and hence $G$, is Roelcke precompact.
\begin{fact} The class of quasi-oligomorphic groups is Borel. \end{fact}
\begin{proof} Recall from Section~\ref{ss:KNT review}  that Roelcke precompactness  is a Borel property of closed subgroups $G $ of $ \S$, and that the operator $\+ M$ is defined for all Roelcke precompact groups $G$.  We claim that for such a group~$G$, 
\bc $G$ is quasi-oligomorphic $\LR$ $\+ M(G) \in \+ B$. \ec
Since $\+ B$ is Borel, this will suffice to establish the fact.

For the implication  ``$\RA$", suppose that $G \cong H$ where $H$ is oligomorphic.  Then $\+ M(G) \cong \+ M(H)  \in \+ B$, so $\+ M(G)  \in \+ B$ as the class $\+ B$ is closed under isomorphism. 

For the implication  ``$\LA$", first recall that $\+ F (\+ M(G)) \cong G$   since $G$ is Roelcke precompact (Prop.~\ref{fact:standard}). Now   suppose that $\+ M(G) \in \+ B$. Then $\+ G( \+ M(G))$ is defined and oligomophic. Since  $\+ G( \+ M(G)) \cong  \+ F( \+ M(G))$, this implies that $G$ is quasi-oligomorphic. 
\end{proof}

Combining  the following   with Theorem~\ref{thm:EC} shows that the isomorphism relation on the class of quasi-oligomorphic groups is essentially countable.
\begin{cor} Isomorphism on the class of quasi-oligomorphic groups is Borel equivalent to isomorphism on oligomorphic groups. \end{cor}
\begin{proof}   If  $G$ is  isomorphic  to an oligomorphic group $H$ then  $\+ M(G) \cong \+ M(H)$ and hence $\+ G(\+ M(G)) \cong H \cong G$. Since $\+ G(\+ M(G))$ is oligomorphic,  the map $G \mapsto \+ G(\+ M(G))$ provides a Borel reduction of the equivalence relation in question to isomorphism of oligomorphic groups.  The converse reduction exists trivially because the two classes are Borel. \end{proof}

\begin{remark} {\rm In contrast, conjugacy of quasi-oligomorphic groups is Borel above isomorphism of oligomorphic groups by the proof of  \cite[Thm.\ 3.1]{Kechris.Nies.etal:18}, and therefore unlikely to be smooth.} \end{remark}

\begin{remark} {\rm We note that the centre $C(G)$ of an oligomorphic group $G$  is finite, and $G/C(G)$ is quasi-oligomorphic in a natural way.  See the post~\cite[Section 4]{LogicBlog:18}, which is joint work with I.\ Kaplan. }
\end{remark}

\subsection{Replacing the $\Pi^1_1$ Axiom \ref{axiom product of filters} \label{son_quatsch}
 by a  first-order axiom} \label{ss:replace fo}
In this subsection, we replace the $\Pi^1_1$ condition in Axiom \ref{axiom product of filters}  (associativity of filter product)  with a   first-order axiom.   {This axiom can be verified  in case that  $M=\+ M(G)$, for any closed subgroup $G$ of $\S$.} The  other axioms  are given by computable  $\omc$ sentences of   finite rank (recall that  we   already replaced  Axiom~\ref{axiom faithful oligomorphic action}   by  Axiom~\ref{axiom faithful Borel version} which is in such a $\omc$ form).  So the class $\+ C$ of coarse groups for oligomorphic groups  is arithmetical.  This class coincides with the class of coarse groups for quasi-oligomorphic groups.


%
The following replaces  Axiom \ref{axiom product of filters}.
Recall that products of appropriate pairs of elements of $M$  are defined immediately after Axiom \ref{axiom products of left and right cosets}. 

\begin{axiom} \label{axiom associativity2} 
If $A\in RC(T)\cap LC(U)$, $B\in RC(U)\cap LC(V)$ and $C\in RC(V)\cap LC(W)$, then $(A\cdot B)\cdot C=A\cdot (B\cdot C)$.  
\end{axiom}  

The products in Axiom \ref{axiom associativity2} are well-defined by Axiom \ref{axiom products of left and right cosets}. 
The axiom holds in $\+M(G)$ since $A\cdot B=AB$ whenever the product $A\cdot B$ is defined. 

\begin{claim}  
The operation $\cdot$ on $\F(M)$ is associative. 
\end{claim} 
\begin{proof} 
It suffices to show $(x\cdot y)\cdot z \subseteq x\cdot(y\cdot z)$ for any $x,y,z\in \F(M)$ since full filters are maximal filters. 

Let  $S\in (x\cdot y)\cdot z$. Find $T\in x\cdot y$ and $C\in z$ with $T C\sqsubseteq S$. Since $T\in x\cdot y$, there are $A\in x$ and $B\in y$ with $A B\sqsubseteq T$. 

We may assume that $A\in LC(U)$, $B\in RC(U)\cap LC(V)$ and $C\in RC(V)$ for some $^*$subgroups $U,V$ by shrinking $A, B, C$ similar as in  the proof of Claim~\ref{product of filters is a filter}. 
In more detail, suppose that $A\in LC(U_0)$ and $B\in RC(\wt U)$. 
Take a $^*$subgroup $U_1\sqsubseteq U_0, \wt U$ by Axiom \ref{basic axioms}(a). 
There is some $A'\in LC(U_1)\cap x$, since $x$ is a full filter, and $A'\sqsubseteq A$ by Axiom \ref{axiom disjointness of cosets}. 
We can similarly find some $B'\sqsubseteq B$ in $RC(U_1)\cap y$. 

Next, suppose that $B'\in LC(V_0)$, $C\in RC(\wt V)$ and take a $^*$subgroup $V_1\sqsubseteq V_0, \wt V$. 
Find $B''\sqsubseteq B'$ in $LC(V_1 )\cap y$ and $C'\sqsubseteq C$ in $RC(V_1)\cap z$. 
Let $U_2$ be a $^*$subgroup with $B''\in RC(U_2)$. Since $B' \in RC(U_1)$ and $B'' \sq B'$, we have  $U_2\sqsubseteq U_1$ by Axiom \ref{axiom cosets for smaller subgroups}. 
There is some $A''\in LC(U_2)\cap x$, since $x$ is a full filter, and $A''\sqsubseteq A'$ by Axiom \ref{axiom cosets for smaller subgroups}. Thus $A''$, $B''$, $C'$ and $U=U_2$, $V=V_1$ are as required. 

We are now ready to  show that $S\in x\cdot (y\cdot z)$.  
Since $A\cdot B\sqsubseteq T$ and $T C\sqsubseteq S$, $(A\cdot B)C\sqsubseteq S$ by monotonicity. Thus $(A\cdot B)\cdot C\sqsubseteq S$ holds by the definition of the product.  Axiom \ref{axiom associativity2} yields  that $A\cdot ( B\cdot C)\sqsubseteq S$.
Now
$A\cdot (B\cdot C)\in x\cdot(y\cdot z)$ holds by Axiom \ref{axiom products of left and right cosets} and the above assumptions on $A, B, C$.      So $S \in x\cdot(y\cdot z)$ as required.
%
%
\end{proof}

As a consequence we have obtained:
\begin{prop} 
\

\n The class $\+ C$ of coarse groups for oligomorphic groups  is arithmetical. \end{prop}  

\def\cprime{$'$} \def\cprime{$'$}

%


\end{document}